\documentclass[a4paper, 10pt, twoside, notitlepage]{article}

\usepackage[utf8]{inputenc}
\usepackage{color}
\usepackage{amsmath} 
\usepackage{amssymb} 
\usepackage{amsthm}
\usepackage{graphicx}
\usepackage{esint}
\usepackage[colorlinks=true,linkcolor=blue]{hyperref}
\usepackage{mathtools}
\usepackage{caption}

\newtheorem{theorem}{Theorem}
\newtheorem{defi}{Definition}[section]

\newtheorem{prop}{Proposition}[section]
\newtheorem{remark}{Remark}[section]

\newtheorem{lem}{Lemma}[section]
\newtheorem*{claim*}{Claim}

\newcommand{\eps}{\varepsilon}
\newcommand{\lb}{\label}
\newcommand{\go}{\rightarrow}

\newcommand{\ee}{\end{equation}}
\newcommand{\be}{\begin{equation}}
\newcommand{\bea}{\begin{eqnarray}}
\newcommand{\eea}{\end{eqnarray}}
\newcommand{\sbea}{\begin{subequations}\begin{eqnarray}}
\newcommand{\seea}{\end{eqnarray}\end{subequations}} 
\newcommand{\ees}{\end{equation*}}
\newcommand{\bes}{\begin{equation*}}
\newcommand{\beas}{\begin{eqnarray*}}
\newcommand{\eeas}{\end{eqnarray*}}

\newcommand{\rf}[1]{(\ref{#1})}

\newcommand{\Z}{\mathbb{Z}}
\newcommand{\N}{\mathbb{N}}
\newcommand{\mR}{\mathbb{R}}

\DeclareMathOperator{\dist}{dist}

\DeclareMathOperator{\Per}{Per}
\DeclareMathOperator{\D}{\Delta}
\DeclareMathOperator{\p}{\partial}

\DeclareMathOperator{\conv}{conv}
\DeclareMathOperator{\inte}{int}
\DeclareMathOperator{\essinf}{essinf}

\begin{document}

\title{Surface Energies Emerging in a Microscopic, Two-Dimensional Two-Well Problem}

\author{Georgy Kitavtsev\footnote{ School of Mathematics, University of Bristol, University Walk, Clifton, Bristol BS8 1TW. {\tt E-mail}: Georgy.Kitavtsev@bristol.ac.uk}, 
Stephan Luckhaus\footnote{Mathematical Institute, University of Leipzig, 04009 Leipzig, Germany. {\tt E-mail}: stephan.luckhaus@math.uni-leipzig.de},
and 
Angkana R\"uland\footnote{Mathematical Institute, University of Oxford, Andrew Wiles Building, Radcliffe Observatory Quarter, Woodstock Road, Oxford OX2 6GG, United Kingdom. {\tt E-mail}: ruland@maths.ox.ac.uk}}
\date{\today}

\maketitle

\begin{abstract}
In this article we are interested in the microscopic modeling of a two-dimensional two-well problem which arises from the square-to-rectangular transformation in (two-dimensional) shape-memory materials. In this discrete set-up, we focus on the surface energy scaling regime and further analyze the Hamiltonian which was introduced in \cite{KLR14}. It turns out that this class of Hamiltonians allows for a direct control of the discrete second order gradients and for a one-sided comparison with a two-dimensonal spin system. Using this and relying on the ideas of Conti and Schweizer \cite{CS06}, \cite{CS06a}, \cite{CS06c}, which were developed for a continuous analogue of the model under consideration, we derive a (first order) continuum limit. This shows the emergence of surface energy in the form of a sharp-interface limiting model as well the  explicit structure of the minimizers to the latter.
\end{abstract}

\tableofcontents

\section{Introduction}

In this article we are concerned with the modeling of a discrete,
two-dimensional square-to-rectangular martensitic phase transition in the
regime of surface energy scaling. Due to their interesting thermodynamical and
mathematical behavior, martensitic phase transitions, being examples of
diffusionless, solid-solid phase transitions, have attracted a large amount of
attention (c.f \cite{B03} and \cite{Mue} for overviews). Hence, a number of
models for these phase transitions, describing them from both microscopic and
macroscopic points of view, exist in the literature. In the present article we
continue to analyze the microscopic discrete model for the
square-to-rectangular martensitic transition which was introduced
in~\cite{KLR14}. In particular, we compare it with its continuous analogues,
which have been considered previously in the literature.

\subsection{Macroscopic, continuum models}
Before describing our microscopic, discrete model, we recall the most commonly used features of \emph{macroscopic, continuum models} for martensitic phase transitions (see e.g~\cite{BJ87,KM92,CO12,C00,CS06} and the references therein). In this context, a classical modeling approach is the analysis of \emph{purely elastic multi-well energies} of the form
\begin{align}
\label{eq:energy_c}
\int\limits_{\Omega} W(\nabla u) dx.
\end{align} 
Here $W:  \mR^{2\times 2}  \rightarrow [0,\infty)$ is an $SO(2)$ invariant, in general non-quasiconvex (bulk) \emph{energy density} which describes the energy cost of deforming a \emph{reference configuration} $\Omega$ into its image configuration $u(\Omega)$ by the \emph{deformation} $u:\Omega  \rightarrow  \mR^{2}$ which is considered under appropriate boundary conditions. It is assumed that the deformation, $u$, and the associated \emph{deformation gradient}, $\nabla u$, are in appropriate Sobolev spaces which are determined by the growth conditions imposed on the energy density $W$.\\
In modeling our phase transitions, we focus on the regime, in which the martensitic phase is favored and $W$ has multiple \emph{energy wells}, i.e. there exist $U_1,\dots,U_m \subset  \mR^{2\times 2}_{sym}$ such that $W(M) =0 $ if and only if $M\in \bigcup\limits_{k=1}^{m} SO(2)U_k$.
Deformations $u$ which (almost everywhere) satisfy $\nabla u \in \bigcup\limits_{k=1}^{m} SO(2)U_k$ are denoted as \emph{exactly stress-free states}. If there are \emph{rank-one connections} between the energy wells, i.e. if for $U_k, U_{k'}$, with $k\neq k'$, there exist a rotation $Q\in SO(2)$ and vectors $a\in  \mR^2 \setminus \{0\}, n\in \mathbb{S}^1$ such that
\begin{align*}
U_k - Q U_{k'} = a \otimes n,
\end{align*}
then examples of stress-free states are provided by so-called \emph{simple laminates}. These are Lipschitz continuous deformations $u$ which only depend on the variable $n\cdot x$ and whose gradient alternates between the two values $U_k$, $QU_{k'}$, e.g.
\begin{align*}
\nabla u(x) \in \left\{ 
\begin{array}{ll}
U_k &\mbox{ if } x\cdot n \geq 0,\\
Q U_{k'} &\mbox{ if } x\cdot n <0.
\end{array}
 \right.
\end{align*}
However, under general boundary conditions, due to the non-quasiconvex nature of the energy density, $W$, exact minimizers of (\ref{eq:energy_c}) do not exist. Instead, in many cases infimizing sequences display highly oscillatory behavior.\\

In order to remedy this non-existence issue and the ``unphysical'', infinitely fine oscillations, \emph{higher order regularizations} are added \cite{KM92,CO12,C00,CS06}, leading to energies like for instance
\begin{align}
\label{eq:energy1}
\int\limits_{\Omega} W(\nabla u) dx + \epsilon \int\limits_{\Omega}|\nabla^2 u|^2 dx \mbox{ with } \epsilon>0.
\end{align} 
These additional higher order contributions are interpreted as \emph{surface energies} since they penalize oscillations and transitions between different energy wells. Due to compactness, in general in the associated spaces minimizers to \rf{eq:energy1} exist and display characteristic length scales (c.f. \cite{KM92}, \cite{CO12}, \cite{C00}).\\
While the basic intention of higher order regularizations always consists of
penalizing too high oscillations, their precise functional form for
macroscopic models is in general not known (from experiments for
instance). Hence a number of different possible regularizations exist, which
range from various kinds of diffuse to sharp interface models. Strikingly,
experiments show the presence of both diffuse and sharp interfaces around twin planes for different materials \cite{BVTA87, BMC09}. 
Hence, in order to answer the question which of these energies is appropriate in which situation,
a more ``first principles'' approach coming from microscopic considerations seems to be desirable. 

\subsection{The microscopic two-well problem}
While the previously described models have had enormous success in predicting material patterns and microstructure, they are all continuum models. As such they are \emph{macroscopic} and ``phenomenological''. In order to develop a more rigorous foundation for these and other macroscopic models in mathematical physics, there has been a great activity in introducing \emph{microscopic} discrete models describing different phenomena in mathematical physics and relating them to their continuum analogues, see e.g. \cite{Br13} and the references therein. 
In these microscopic models it is assumed that the elastic sample is given as a deformation of a ground state (atomic) lattice, e.g. $(n^{-1}\Z)^2$ or subsets thereof. This deformation is energetically described by a Hamiltonian, i.e. a sum of local energies originating from
interaction of the ``atoms'' involved in the microscopic sample.
The described discrete models have been thoroughly analyzed in one-dimension (e.g. for elastic chains) \cite{BC07}. As in the continuous analogue, vectorial problems are less well understood. Here various approaches are pursued \cite{LM10}, \cite{TF02}, \cite{BBL02}, \cite{AC04}, \cite{BS12}, \cite{Ro12}. Moreover, a direct comparison of the scaling behavior of a discrete and a continuous model of a two-dimensional two-well problem is given in \cite{Lo06}, \cite{Lo09}.
In the context of the vectorial set-up in martensitic phase transitions a key (mathematical) difficulty which distinguishes it from the one-dimensional case is the presence of a \emph{continuum} of energy wells, i.e. $\bigcup\limits_{j=1}^{k} SO(2)U_j$.\\

In the sequel, we address a specific two-dimensional martensitic phase transition, the
square-to-rectangular phase transition, from a microscopic point of view. In the following
subsections, we introduce our precise set-up based on the discrete two-well Hamiltonian (with $SO(2)$ symmetry) and present our main results.

\subsubsection{Setting}
\label{sec:setting}
In this section we describe the basic set-up of our discrete two-well
problem. We define the underlying domains and function spaces and explain our explicit model Hamiltonian. 
In the sequel, we seek to model the two-dimensional square-to-rectangular phase transition in the martensitic phase in which the variants of martensite constitute the energy wells. For this purpose we introduce the following energy wells 
\begin{align}
\label{eq:wells}
K:=SO(2)U_0  \cup SO(2)U_1,  \mbox{ where } U_0 = \begin{pmatrix} a & 0 \\  0 & b \end{pmatrix}, \ U_1 = \begin{pmatrix} b & 0 \\ 0 & a \end{pmatrix}, \ a\neq b.
\end{align}
Although for the mathematical treatment of our problem this is not necessary, we restrict ourselves to the most relevant physical situation of volume preserving transformations. This corresponds to the assumption $ab=1$. 
In this study, for convenience of notation, we denote by $c$ any positive
constant depending only (if not stated explicitly otherwise) on the lattice parameters $a$ and $b$. Moreover, we often use the special constant
\be
\bar{c}:=\dist(SO(2)U_0,\,SO(2)U_1).
\lb{cb}
\ee
We recall that for every deformation $U\in SO(2)U_0$ there are exactly two rank-one connected matrices in $SO(2)U_1$, i.e. there exist (exactly) two rotations $Q, \tilde{Q}$ such that
\begin{equation}
\begin{split}
U_0-QU_1&=\sqrt{2}\frac{a^2-b^2}{a^2+b^2}\begin{pmatrix} a\\-b \end{pmatrix}\otimes\frac{1}{\sqrt{2}}\begin{pmatrix}1\\ 1\end{pmatrix},\\ 
U_0-\tilde{Q}U_1&=\sqrt{2}\frac{a^2-b^2}{a^2+b^2}\begin{pmatrix} a\\b \end{pmatrix}\otimes\frac{1}{\sqrt{2}}\begin{pmatrix}1\\ -1\end{pmatrix}.
\label{eq:ROC}
\end{split}
\end{equation} 
Macroscopically, we expect that a shape memory alloy with these wells can form two variants of simple laminates (without bulk energy cost), where the normals to the jump planes are either given by $\begin{pmatrix} 1 \\ 1 \end{pmatrix}$ or $\begin{pmatrix} 1 \\ -1 \end{pmatrix}$. Our first compactness result (c.f. Proposition  \ref{prop:rig2}) shows that indeed this \emph{macroscopic} expectation can be justified by starting from a \emph{microscopic} point of view and passing to the corresponding continuum limit. 
Moreover, we note that due to the rank-one connections between the wells, $U_0$ and $QU_1$ are both rank-one connected with their convex combinations
\begin{align}
\label{eq:Fl}
F_{\lambda}:= \lambda U_0 + (1-\lambda) Q U_1,
\end{align}
where $Q$ is the rotation from (\ref{eq:ROC}) and $\lambda\in (0,1)$.\\

In order to prepare the definition of our model Hamiltonian which is formed by a sum of local energies $h$ having the set $K$ as the set of their local minimizers, we  first define the precise set-up regarding the underlying domains and function spaces (c.f. Figure \ref{fig:setup}):

\begin{figure}
\includegraphics[width=\textwidth]{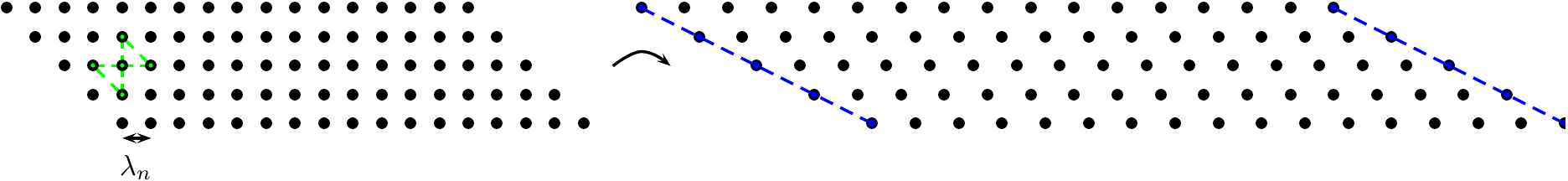}
\caption{The set-up of our problem. On the left the reference domain $\Omega_n$ is depicted. The green triangles correspond to two of the grid triangles $\Delta^{n,\pm}_{ij}$. $\Omega_n$ is mapped into an image configuration, shown on the right. In the image configuration, admissibility as formulated in (\ref{eq:bc}) has to be preserved. This includes the non-interpenetration condition but also the attainment of the correct boundary data along the blue lines.}
\label{fig:setup}
\end{figure}

\begin{defi}[Domains, deformations, admissibility]
\label{defi:setup}
In the sequel, we consider the following domains:
\begin{itemize}
\item For $x\in \mR^2$, $d,l\in \mR$, we define the translated parallelograms
\begin{equation}
\label{eq:parallel}
\begin{split}
\Omega_{d,l}^{+}(x)&:= \conv\{(-d-l,l),(d-l,l),(d+l,-l),(l-d,-l)\} + x,\\
\Omega_{d,l}^{-}(x)&:= \conv\{(l-d,l),(l+d,l),(-l-d,-l),(d-l,-l)\} + x.
\end{split}
\end{equation}
If $x=(0,0)$, we also omit the point in the notation and simply write $\Omega_{d,l}^{+}$, $\Omega_{d,l}^{-}$.
\item Using the previous notation, we set $\Omega:=\Omega_{4,1}^{+}$ and $\Omega_n:= \Omega \cap (n^{-1}\Z)^2$.
\item In the sequel, we work on the following triangles:\\
$\Delta_{i,j}^{n,+}:= \conv\{(i/n,j/n),(i/n,(j+1)/n),((i+1)/n,j/n)\}$, \\
$\Delta_{i,j}^{n,-}:= \conv\{(i/n,j/n),((i-1)/n,j/n),(i/n,(j-1)/n)\}$.
\item $G^n(\Omega_n)$ denotes the set of all edges involved in the triangles $\Delta_{i,j}^{n,\pm}\subset \Omega_n$. With slight abuse of notation, we also refer to ``the grid $\Omega_n$'' in the sequel, by which we mean the pair $(\Omega_n,\Delta_{i,j}^{n\pm})$ for $(i,j)\in n\Omega_n$.
\end{itemize}
Moreover, for any given set $M\subset \mR^2$ and any $n\in \N$ we define
\begin{align*}
nM:= \{(i,j): (i/n,j/n)\in M\}.
\end{align*}
In particular, this defines the sets $n\Omega$ and $n\Delta_{ij}^{n,\pm}$.\\
On the respective domains, we consider associated deformations $u\in C(\Omega,\, \mR^2) \cap H^1(\Omega,\, \mR^2)$ and denote their values on lattice points as
\begin{align}
\label{eq:evalI}
u^{i,j}:=u(i/n,\,j/n)\quad\text{for}\quad (i,j)\in n\Omega. 
\end{align}
Additionally we restrict ourselves to \emph{admissible} lattice deformations $u$  satisfying a \emph{non-interpenetration} condition,
i.e. on any domain $\Omega_{d,l}^{\pm}$ under consideration $u\in\mathcal{A}_{n \Omega_{d,l}^{\pm}}$, where
\begin{equation}
\label{eq:non}
\begin{split}
\mathcal{A}_{n\Omega_{d,l}^{\pm}}& := \{ v\in C(\Omega_{d,l}^{\pm},\, \mR^2) \cap H^1(\Omega_{d,l}^{\pm},\, \mR^2):\\
& \quad \mbox{for all } n\Delta_{ij}^{n,\pm}\in\Omega_{d,l}^{\pm},\quad \det(v^{ij}-v^{kl},v^{rs}-v^{kl})>0,\\
& \quad \mbox{where } (k,l),\,(r,s)\in \Z^2\mbox{ are adjacent grid vertices in } n\Delta_{ij}^{n,\pm}\}.
\end{split}
\end{equation}
Finally, we impose \emph{boundary conditions} prescribed by the matrices from \rf{eq:Fl} on admissible deformations considered on the whole domain $\Omega$, 
i.e. for any admissible deformation defined on the whole domain $\Omega$, we demand that $u\in\mathcal{A}_n^{F_\lambda}$, where
\begin{equation}
\label{eq:bc}
\begin{split}
\mathcal{A}_n^{F_\lambda}& := \{ v\in \mathcal{A}_n(\Omega):\\
&\quad v^{ij} = F_{\lambda}(i/n,j/n)   \mbox{ for all } i+j \leq -4n,\\
& \quad v^{ij} = F_{\lambda}(i/n,j/n) + c \mbox{ for some } c\in \mathbb{R}^2 \mbox{ for all } i+j \geq 4n\}.
\end{split}
\end{equation}
\end{defi}
Let us comment on these notions. We remark that we can easily switch from functions $u:\Omega  \rightarrow  \mR^2$ to \emph{grid functions} $u^{ij}:n \Omega  \rightarrow  \mR^2$ using the definition given in (\ref{eq:evalI}) above. Conversely, starting from a discrete lattice function $u^{ij}:n \Omega  \rightarrow  \mR^2$, we can pass to a function $u\in C(\Omega,\, \mR^2)$ by
piecewise affine interpolation on the triangles $n\Delta_{i,j}^{n,\pm}$. Hence most of the functions $u\in  \mathcal{A}_n^{F_\lambda}$ in our applications will be piecewise affinely interpolated lattice functions satisfying (\ref{eq:bc}). Thus, with slight abuse of notation, in the sequel we will use the phrase ``let $\{u_n\}_{n\in\N}\subset \mathcal{A}_n^{F_{\lambda}}$ be a sequence of piecewise affine functions on $\Omega_n$'' to denote a sequence of admissible functions that is affine on all of the grid triangles $\Delta_{i,j}^{n,\pm}\subset \Omega$.\\
In the context of our admissible lattice functions, the non-interpenetration condition contained in (\ref{eq:non}) corresponds to requiring that the labeling of the lattice triangles $n \Delta_{i,j}^{n,\pm}$ is not reversed or interchanged under the deformation $u$. Hence, it can be interpreted as a local invertibility constraint for piecewise affine deformations $u$ on each of the lattice triangles $\Delta_{i,j}^{n,\pm} \subset \Omega$ (c.f. the recent article of Braides and Gelli, \cite{BG15}, for a criticism of this in the context of discrete-to-continuum fracture mechanics).

\begin{remark}
\label{rmk:interpol}
By applying the conventions from Definition 1.1 and in particular by interpreting $u:\Omega \rightarrow \mR^2$ as the above described piecewise affine interpolation of the lattice deformation $u^{i,j}$, the usual spatial gradient, $\nabla u$ of the deformation is well-defined and satisfies $\nabla u \in PC(\Omega, \mR^{2\times 2})$. Here $PC(\Omega, \mR^{2\times 2})$ denotes the space of piecewise constant matrix valued functions.
In addition, we further agree on working with the following Lebesgue representative for the equivalence class of our deformation gradient $\nabla u$: On the edges of the lattice triangles $\partial (\Delta_{i,j}^{n,\pm})$ we define the gradient of $u$ to be equal to its value in the interior of the corresponding triangle $\Delta_{k,l}^{n,+}$ which contains the considered edge. Also at the lattice point $(i/n,j/n)\in \Omega_n$ the gradient is identified with the one on $\Delta_{i,j}^{n,+}$. Using this convention, the abbreviations
\beas
&&\nabla u^{i,j}:=\nabla u(i/n,\,j/n),\quad \partial_s u^{i,j}:=\partial_s u(i/n,\,j/n)\quad\mbox{for}\quad s=1,2,
\eeas
which are used in the sequel, are properly defined.\\
In Propositions \ref{prop:twell} and \ref{prop:reduc} we will also deal with functions $\phi^{ij}_{u}$ of $\nabla u\in PC(\Omega, \mR^{2\times 2})$, which are compositions of Lipschitz functions $\phi$ and $\nabla u$. Consequently, these are piecewise constant themselves. In this context, we will work with discrete gradients which we denote by $\nabla_n \phi^{ij}_{u}$. Here, for any lattice function $f^{ij}:\Omega_n \rightarrow \mR$, we set
\begin{align}
\label{eq:discr_grad}
\nabla_n f^{ij} = (n( f^{i+1,j}- f^{i,j}), n(f^{i,j+1}-f^{i,j}))^t.
\end{align}
\end{remark}

\begin{remark}
Note that in order to avoid additional technicalities connected with boundary effects and in order to keep the leading order of the bulk elastic energy zero in all considerations below we define our reference domain $\Omega$ as the parallelogram from Definition 1.1 (with sides orthogonal to one of the normals of the rank-one connections from \rf{eq:ROC}). 
\end{remark}

Keeping these conventions in mind, we proceed with the definition of our model Hamiltonian, which had already been previously introduced in~\cite{KLR14}.

\begin{defi}[Model Hamiltonian, I] 
\label{defi:HamiltonianI}
Let $u\in C(\Omega, \  \mR^2) \cap H^1(\Omega,\, \mR^2)$. Then the \emph{model Hamiltonian} on the lattice $n\Omega$ is defined as
\begin{eqnarray}
\tilde{H}_n(u)&=&\sum\limits_{(i,j)\in n\Omega}\tilde{h}_u^{i,j}:=\sum\limits_{i,j\in n\Omega} \lambda_{n}^2 \tilde{h}\left(\frac{u^{ij}-u^{i\pm 1j}}{\lambda_{n}},\frac{u^{ij}-u^{ij \pm 1}}{\lambda_{n}} \right)\nonumber\\
&:=&\sum_{(i,j)\in n\Omega}\lambda_{n}^{2}\left[\left(\left(\frac{u^{ij\pm 1}-u^{ij}}{\lambda_n} \right)^2- a^2 \right)^2+\left(\left(\frac{u^{i\pm 1j}-u^{ij}}{\lambda_n} \right)^2- b^2 \right)^2  \right. \nonumber\\
&&\left.+ \left|\left(\frac{u^{ij\pm 1}-u^{ij}}{\lambda_n},\frac{u^{i\pm 1j}-u^{ij}}{\lambda_n} \right) \right| \right]\times\nonumber\\
&\times&\left[\left(\left(\frac{u^{ij\pm 1}-u^{ij}}{\lambda_n} \right)^2-  b^2 \right)^2+\left(\left(\frac{u^{i\pm 1j}-u^{ij}}{\lambda_n} \right)^2- a^2 \right)^2  \right.
\nonumber\\ 
&&\left.+\left|\left(\frac{u^{ij\pm 1}-u^{ij}}{\lambda_n},\,\frac{u^{i\pm 1j}-u^{ij}}{\lambda_n} \right) \right| \right ],
\label{HD}
\end{eqnarray}
where $\lambda_n:= n^{-1}$ and $(\cdot,\,\cdot)$ denotes the scalar product in $\mR^2$.
In the above definition and below, we use a summation agreement: The sign $\pm$ in a term indicates that the latter should be replaced by the sum of the respective terms having all possible sign combinations, e.g.
\begin{eqnarray*}
&&\left((u^{ij\pm 1}-u^{ij})^2-(\lambda_n  a)^2 \right)^2:=\\
&&:=\left((u^{ij + 1}-u^{ij})^2-(\lambda_n  a)^2 \right)^2+\left((u^{ij - 1}-u^{ij})^2-(\lambda_n  a)^2 \right)^2,\\[2ex]
&& (u^{ij\pm 1}-u^{ij},\,u^{i\pm 1j}-u^{ij}):=\\
&&:=(u^{ij-1}-u^{ij},\,u^{i-1j}-u^{ij})+(u^{ij-1}-u^{ij},\,u^{i+1j}-u^{ij})\\
&&+(u^{ij+1}-u^{ij},\,u^{i-1j}-u^{ij})+(u^{ij+1}-u^{ij},\,u^{i+1j}-u^{ij}).
\end{eqnarray*}
\end{defi}

\begin{remark}[Boundary conditions]
\label{rmk:boundary_cond}
We stress that in our model we impose ``hard'' boundary conditions in the form of (\ref{eq:bc}). Already at this stage we emphasize that they are ``seen'' by our Hamiltonian, as for points $(i,j)$ on the layers  with $i+j=-4n$, $i+j=4n$ the Hamiltonian still takes the left and right neighbors of these into account. On these the boundary conditions have already been prescribed. This will give rise to boundary layer energies (c.f. Theorem \ref{Th1}).
\end{remark}

The Hamiltonian in Definition \ref{defi:HamiltonianI} is constructed in such a way that the matrices from (\ref{eq:wells}) indeed form its energy wells, i.e. $\tilde{H}_n(u)\geq 0$ for all admissible $u$ and $\tilde{H}_n(u)=0$ if and only if $\nabla u \in SO(2)U_0$ or $\nabla u \in SO(2)U_1$ on the whole of $\Omega$.
On the level of the local energies $\tilde{h}^{i,j}_{u}$ the deviation from the wells is measured by penalizing deformations which do not map horizontal and vertical unit line segments onto line segments of either the lengths $a$ or $b$. Physically, this corresponds to two-body interactions between the five neighboring atoms $(i/n,j/n)$, $(i\pm 1/n,j/n)$ and $(i/n,j\pm 1/n)$. Moreover, the local energy $\tilde{h}^{ij}_u$ favors deformations which enforce that orthogonal line segments are mapped to orthogonal line segments, i.e. deviations from orthogonality of the pairs $(i\pm 1/n,j/n)$ and $(i/n,j\pm 1/n)$ are penalized.
The latter is physically achieved by three-body interactions measuring the angle between $u^{i\pm 1,\,j}$ and $u^{i,\,j\pm 1}$, respectively.
We emphasize that the condition on the angles is necessary in modeling the deformation of an elastic body, as otherwise no shear resistance would be present.

\begin{remark}
\label{rmk:bracket}
Using the notation from Definition \ref{defi:setup}, we can rewrite the brackets in the definition of the Hamiltonian (\ref{HD}) in terms of lengths and angles of the horizontal and vertical derivatives of the deformations. For instance, the first bracket in (\ref{HD}) turns into
\begin{equation}
\label{eq:bracket1}
\begin{split}
&\bar{h}_{U_0}(\partial_1 u^{i,j},\partial_1 u^{i-1,j},\partial_2 u^{i,j},\partial_2 u^{i,\,j-1}):=(|\partial_1 u^{i,j}|^2 - a^2)^2+ (|\partial_1 u^{i-1,j}|^2-a^2)^2 \\
&+ (|\partial_2 u^{i,j}|^2 - b^2)^2 + (|\partial_2 u^{i,j-1}|^2-b^2)^2+  |(\partial_1 u^{i\pm 1,j},\partial_2 u^{i,j\pm 1})|.
\end{split}
\end{equation}
Hence the local energy density $\tilde{h}:\, \mR^2 \times \mR^2 \times \mR^2 \times \mR^2 \go  \mR$
\begin{align*}
(\partial_1 u^{i,j},\partial_1 u^{i-1,j},\partial_2 u^{i,j},\partial_2 u^{i,\,j-1})&\mapsto \tilde{h}(\partial_1 u^{i,j},\partial_1 u^{i-1,j},\partial_2 u^{i,j},\partial_2 u^{i,\,j-1})\\
& \quad :=\bar{h}_{U_0}(\partial_1 u^{i,j},\partial_1 u^{i-1,j},\partial_2 u^{i,j},\partial_2 u^{i,\,j-1})\\
& \quad \quad  \times  \bar{h}_{U_1}(\partial_1 u^{i,j},\partial_1 u^{i-1,j},\partial_2 u^{i,j},\partial_2 u^{i,\,j-1})
\end{align*}
can also be regarded as a function of the lengths and angles formed by the corresponding partial derivatives of the deformation.
Here
\begin{align*}
\bar{h}_{U_1}(\partial_1 u^{i,j},\partial_1 u^{i-1,j},\partial_2 u^{i,j},\partial_2 u^{i,\,j-1}):= \bar{h}_{U_0}(\partial_2 u^{i,j},\partial_2 u^{i,j-1},\partial_1 u^{i,j},\partial_1 u^{i-1,\,j}).
\end{align*}
This also serves as a guiding intuition in defining a more general class of Hamiltonians for which our results are valid (see Definition \ref{defi:Hamiltonian2} below).\\
In this sense the density $\tilde{h}$ in the Hamiltonian from Definition \ref{defi:HamiltonianI} (and later also the ones from Definitions \ref{defi:HamiltonianII} and \ref{defi:Hamiltonian2}) can be viewed as a Lipschitz continuous function. In order to construct the Hamiltonian it is composed with the discrete function $u^{ij}$ (or, depending on which point of view is more suitable in the respective situation, with the piecewise constant function $\nabla u$). \\
For the individual brackets of the Hamiltonian from Definition \ref{defi:HamiltonianI} we introduce the notation
\begin{align*}
\bar{h}^{ij}_{u, U_0}&:= \bar{h}_{U_0}(\partial_1 u^{i,j},\partial_1 u^{i-1,j},\partial_2 u^{i,j},\partial_2 u^{i,\,j-1}),\\
\bar{h}^{ij}_{u, U_1}&:= \bar{h}_{U_1}(\partial_1 u^{i,j},\partial_1 u^{i-1,j},\partial_2 u^{i,j},\partial_2 u^{i,\,j-1}).
\end{align*}
This notation is used in the Step 2 of the proof of the $\Gamma$-convergence result in Section \ref{sec:Gamma}.
\end{remark}

Apart from this geometric interpretation of the energy density, the definition of $\tilde{h}$ also implies the following pointwise control on the distance of $\nabla u$ to the energy wells $K$.

\begin{lem}[Lower bound]
\label{lem:lower}
Let $u\in \mathcal{A}_n^{F_\lambda}$. Then for each $(i,j)\in n\Omega$ the following bound holds:
\begin{align*}
\tilde{h}^{i,j}_u \geq c \dist(\nabla u^{i,j}, K)^2.
\end{align*}
In particular,
\begin{align*}
H_n(u) \geq c \int\limits_{\Omega} \dist(\nabla u, K)^2 dx.
\end{align*} 
\end{lem}

\begin{proof}
The proof follows immediately from noticing that for $\nabla u^{i,j} \in GL(2, \mR)_+ $ 
\begin{align*}
&\min\{||\partial_1 u^{i,j}| - a|+ ||\partial_2 u^{i,j}| - b|, ||\partial_1 u^{i,j}| - b|+ ||\partial_2 u^{i,j}| - a|\}^2\\
& \quad + |(\partial_1 u^{i,j},\,\partial_2 u^{i,j})| \geq c \dist(\nabla u^{i,j},\, K)^2.
\end{align*}
\end{proof}

We further remark that the Hamiltonian defined in Definition \ref{defi:HamiltonianI} is of mixed $L^2-L^8$ growth:
If on $\Delta_{ij}^{n,\pm}$ the gradient $\nabla u$ is close to the wells, the local energy $h$ is comparable to $\dist^2(\nabla u_n, SO(2)U_0 \cup SO(2)U_1)$, while if $\nabla u$ is at a finite distance from the wells, the energy controls the $L^8$ norm of $\nabla u$. Hence, in total
\begin{align*}
\tilde{H}_{n}(u_n) \geq c\int\limits_{\Omega} \max\{ \dist^8 (\nabla u, K), \dist^2(\nabla u, K) \}.
\end{align*}
In order to avoid technical difficulties with the mixed growth behavior at infinity, we truncate the energy density for large gradient values.

\begin{defi}[Hamiltonian, II]
\label{defi:HamiltonianII}
Let $u\in  \mathcal{A}_n^{F_\lambda}$. Let $k:(\mR^{2})^4 \rightarrow \mR$ be a Lipschitz continuous function satisfying the bound
\bes
c_1|F|^2 \leq k(F) \leq c_2 |F|^2,
\ees
for any $F\in (\mR^{2})^4$ and for some universal constants $c_1, c_2 >0$. Using this function, we define a modified energy density as
\begin{align}
\label{eq:mod}
h_u^{i,j}:=\gamma (|U^{i,j}|) \tilde{h}_u^{i,j} + (1-\gamma)(|U^{i,j}|) k(U^{ij}),
\end{align}
with $U^{i,j}:=(\partial_1 u^{i,j},\partial_1 u^{i-1,j},\partial_2 u^{i,j},\partial_2 u^{i,\,j-1})$ and a cut-off function $\gamma\in C^\infty(\mR,\mR)$ which is chosen such that $\gamma(F)=1$ for all $|F| \leq 10(\bar{c}+1)$ and $\gamma (F)= 0$ for $|F|\geq 20(\bar{c}+1)$. Here the constant $\bar{c}$ is the one from \rf{cb}.
Using this, we define the \emph{(final) model Hamiltonian} as 
\begin{align*}
H_n(u)&:=\sum\limits_{(i,j)\in n\Omega} \lambda_{n}^2 h^{i,j}_u. 
\end{align*}
\end{defi}

\begin{remark}
We remark that the energy density from Definition \ref{defi:HamiltonianII} in particular satisfies $L^2$ bounds at infinity and for each $(i,\,j)\in n\Omega$ the global estimate
\begin{align}
\lb{e:lb}
\dist^2(\nabla u^{i,j}, SO(2)U_0  \cup SO(2)U_1)  \lesssim h^{i,j}_u \lesssim \dist^2(\nabla u^{i,j}, SO(2)U_0  \cup SO(2)U_1)
\end{align}
holds.
\end{remark}

In concluding this section we stress that the following results are not only valid for our model Hamiltonian from Definition  \ref{defi:HamiltonianII} but hold for the following more general class of discrete Hamiltonians.

\begin{defi}[General class of Hamiltonians]
\label{defi:Hamiltonian2}
For any $u\in\mathcal{A}_n^{F_\lambda}$ let $h_u^{ij}$ be the density from Definition \ref{defi:HamiltonianII} and let $\hat{h}_{\cdot} \in C^{0,1}((\mR^2)^4,\mR)$  be such that (\ref{e:lb}) is satisfied and the following lower bound holds:
\begin{align}
\label{eq:lower_a}
\hat{h}_u^{i,j} \geq c_1 h_u^{i,j} ,
\end{align}
for each $(i,j)\in n\Omega$ and some positive constant $c_1$. Here
the abbreviation $\hat{h}^{ij}_u$  is used as above to denote the pointwise evaluation $\hat{h}_u(i/n,j/n)$ at $(i,j)\in n\Omega$ (analogously as the evaluations defined in Remark \ref{rmk:bracket}). We set
\begin{align*}
H_n(u):= \sum\limits_{(i,j)\in n\Omega} n^{-2} \hat{h}^{i,j}_u.
\end{align*}
\end{defi}
We emphasize two important properties of this class of Hamiltonians: These are the lower bound from Lemma \ref{lem:lower} (which follows from that for $h^{ij}_u$) and the lower bound (\ref{eq:lower_a}), which allows us to invoke the comparison arguments from Appendix \ref{sec:spin}. The lower bound (\ref{eq:lower_a}) contains the origin of the surface energies which are more closely analyzed in Section \ref{surface}.\\
As there are only small modifications in the proofs of the following results, we always carry them out for our model Hamiltonian $H_n$ from Definition \ref{defi:HamiltonianII} and leave the corresponding modifications for the general class to the reader.

\subsubsection{Main results}
Our main objective in this article is an analysis of the discrete square-to-rectangular phase transition in the regime of surface energy scaling. 
In this context we will analyze the microscopic Hamiltonians from Definitions \ref{defi:HamiltonianII} and \ref{defi:Hamiltonian2} on ``low energy deformations''. More precisely, as in \cite{KLR14} in the sequel we study sequences of admissible deformations $\{u_n\}_{n\in \N}$ for which 

\begin{align}
\label{eq:surfscaling}
H_{n}(u_n)\leq \frac{C}{n}, 
\end{align}
for some (in $n$) uniform constant $0<C<\infty$. 
As an immediate property of this scaling, we observe that, since the Hamiltonian controls the distance of $\nabla u$ to the wells, we in particular obtain $ \dist(\nabla u, SO(2)U_0 \cup SO(2)U_1)  \rightarrow 0$ in measure for deformations satisfying (\ref{eq:surfscaling}).\\
As in the case of atomic chains, which were investigated in \cite{KLR14}, we expect that this scaling in $n$ (in combination with the boundary conditions given by  (\ref{eq:Fl})) yields deformations which are locally simple laminates in the limit $n \rightarrow \infty$ (see Proposition \ref{prop:rig2} below).\\

In this context our main interest is driven by the modeling side of the problem and by the question of whether the discrete problem can be regarded as an ``equivalent'' of the continuous regularization.
Mathematically, our analysis is strongly based on the fundamental ideas introduced in the treatment of the continuum version of the two-well problem by Conti and Schweizer \cite{CS06}, \cite{CS06a}, \cite{CS06c}.
Let us also mention that these methods differ substantially from the ones introduced in our preceding paper \cite{KLR14}, where, after a special reduction of the Hamiltonian from Definition \ref{defi:HamiltonianI} onto one-dimensional chains, low energy deformations were treated both analytically and numerically. Due to the presence of ``atomic chains'', additional structural conditions had been exploited in that context.\\
 
As the main result of this article, we prove that in the surface energy scaling regime and in the continuum limit, i.e. as $n \rightarrow \infty$, there exists a limiting surface energy which resembles the analogous limiting energy from the continuous set-up \cite{CS06}. This result is formulated 
precisely in the following theorem.
\begin{theorem}
\lb{Th1}
Let $H_n(u)$ be as in Definition  \ref{defi:HamiltonianII}. Then in the sense of $\Gamma$-limits with respect to the $L^1$ topology on $\mathcal{A}_n^{F_\lambda}$
\begin{align*}
n H_n  \rightarrow E_{surf},
\end{align*}
where
\begin{align*}
E_{surf}(u_0):= \left\{
\begin{array}{ll}
\int\limits_{J(\nabla u_0)} \bar{C}(\nabla u_0(x-,0), \nabla u_0(x+,0)) d\mathcal{H}^1, \mbox{ if } u_0 \mbox{ is a piecewise} \\
\quad \quad \mbox{ affine deformation satisfying the boundary condition } \\ 
\quad \quad \mbox{ in (8) for all $n\in \N$ and with gradient } \nabla u_0 \in BV\\
\quad \quad \mbox{ such that } \nabla u_0 \in SO(2)U_0 \cup SO(2)U_1 \mbox{ a.e. in }\Omega, \\[3ex]
\infty, \mbox{ else}.
\end{array}  \right.
\end{align*}
Here $J_{\nabla u_0}$ denotes the jump set of $\nabla u_0$ and $\bar{C}(\cdot, \cdot)$ is the limiting energy density defined in Definition \ref{defi:limiting energy} below.
\end{theorem}

Let us comment on the result of Theorem \ref{Th1} and its relation to \cite{Lo09}: Similarly as in \cite{Lo09} our analysis is motivated by understanding the relation between the discrete and continuous regularizations of the square-to-rectangular phase transition. In this context we are in particular interested in studying the origins of surface energies. The article \cite{Lo09} compares the scaling of infimizers of a functional of the type (\ref{eq:energy1}) with a discretization of (\ref{eq:energy_c}) and shows that the corresponding scaling behaviors coincide. In particular, (for infimizers) this allows to switch between the discrete and continuum functional (up to giving up constants) and to transfer bounds from the discrete to the analogous continuum model. In principle, this would permit us to establish compactness properties for sequences for the discrete (minimizing) sequences from the compactness properties of their continuous analoga. However, due to the loss of the constants, \cite{Lo09} does \emph{not} imply our $\Gamma$-convergence result. Instead of exploiting the result of \cite{Lo09}, we give an independent proof of the compactness properties, since in our discrete setting this step is simplified by a comparison with a spin system (due to the presence of next-to-nearest neighbor interactions). Having established compactness, our proof then follows the ideas outlined by Conti and Schweizer \cite{CS06}, \cite{CS06a}, \cite{CS06c} adapted to our discrete set-up.\\
Let us further note that while Theorem \ref{Th1} does not explain the different diffuse and sharp interface features observed in experiments, it does show that as in the one-dimensional case and as in \cite{Lo06}, \cite{Lo09}, discrete two-well energies such as in Definitions \ref{defi:HamiltonianII} and \ref{defi:Hamiltonian2} naturally lead to higher order regularizations. 
If understood on the level of finite, but large sample sizes they might even give indications for the experimentally observed behavior. The analogy between the discrete and the continuous setting is highlighted in Table \ref{tab:1}. It provides a natural correspondence between continuum objects considered by Conti and Schweizer \cite{CS06}, \cite{CS06a}, \cite{CS06c} and our discrete setting introduced in the previous paragraph.

\begin{table}[t]
\centering
\begin{tabular}[c]{|c|c|c|c|}
 \hline\hline 
&   continuum, macroscopic &  discrete, microscopic  \\
\hline \hline
energies  &   \ $\int\limits_{\Omega} \frac{1}{\epsilon} W(\nabla u) + \epsilon |\nabla^2 u|^2 dx \leq C,  $ &  \ $\sum\limits_{i=-n}^n\sum\limits_{j=-n}^n \frac{1}{n^2} h_{u_n}^{ij}\leq  \frac{C}{n} $,  \\
&   \ $\epsilon  \rightarrow 0 $&  \ $n  \rightarrow \infty$ \\
 \hline
 &   &  \\
regularity of  &   &  \ $u \in C^{0,1}(\Omega, \mR^2)$ is piecewise    \\
the deformation &   \ $u \in W^{2,2}(\Omega, \mR^2)$  & \  affine on the underlying  \\
&   &  \ lattice (after interpolation)  \\
\hline
& & \\
surface energies &   \ $\int\limits_{\Omega}  \epsilon |\nabla^2 u|^2 dx \leq C$ &   \ $\sum\limits_{i=-n}^n\sum\limits_{j=-n}^n \frac{1}{n^2}|\nabla^2_n u|^2 \leq C n$  \\
\hline
 & &\\
scaling  &   $ \ \epsilon$ &  $ \ \frac{1}{n}$ \\
parameter &  &   \\
 \hline\hline 
\end{tabular}
\caption{Comparison of the relevant continuous and discrete quantities.}
\label{tab:1}
\end{table}

\subsection{Organization of the article}
The remainder of the article is organized as follows. In Section  \ref{sec:rigidity} we derive compactness (Proposition \ref{prop:rig0}) and rigidity properties (Propositions \ref{prop:rig2}, \ref{prop:twell}) of the sequences of deformations $u_n$ obeying the surface energy scaling (\ref{eq:surfscaling}).
In Section \ref{surface} we describe the limiting surface energies and prove Theorem \ref{Th1}. Important auxiliary results are proved in the Appendices: In Appendix \ref{sec:spin} we provide a mapping of our discrete two-well problem to a spin system. This yields a one-sided estimate of the original Hamiltonian from below.
These results are crucially used in our proof of the compactness results of Propositions \ref{prop:rig0}, \ref{prop:rig2}. 
In Appendix \ref{sec:2derivative} we show that our discrete Hamiltonian $H_n$ provides upper bounds of the discrete second derivatives of admissible deformations $u$. 
While the latter bounds are not actually necessary for our argument, they are included as an illustration of the comparability of our discrete and the continuous model from \cite{CS06}. Last but not least in Appendix \ref{app:Coarea} and Appendix \ref{app:alg} we give sketches of the proofs of the discrete coarea formula and the well-definedness of the algorithm yielding the perturbed grid in the proof of Proposition \ref{prop:reduc}, Step 4a in Section \ref{surface}.

\section{Rigidity}
\label{sec:rigidity}
In this section we prove various rigidity estimates. On the one hand they yield compactness properties (c.f. Propositions \ref{prop:rig0}, \ref{prop:rig2}) for admissible sequences $\{u_n\}_{n\in \N}$ which obey the energy bound (\ref{eq:surfscaling}). On the other hand, adapting to our discrete setting the ideas of Conti and Schweizer \cite{CS06}, \cite{CS06a}, \cite{CS06c}, we show finer rigidity estimates for sequences, which, in addition to (\ref{eq:surfscaling}), also satisfy the smallness condition \rf{eq:majority_phase} for a local \emph{one-well} energy (see Proposition \ref{prop:twell} below). The latter estimates play a crucial role for the cutting mechanism which is used for the construction of the
recovery sequence in the $\Gamma$-$\limsup$ inequality (c.f. Section \ref{sec:Gamma}).

\subsection{Compactness}
\lb{sec:compact}
In this section we exploit the information from Appendix \ref{sec:spin}, in order to prove rigidity of the limiting deformation fields (c.f Propositions \ref{prop:rig0}, \ref{prop:rig2}). 
We begin with the following auxiliary result:

\begin{lem}
\label{lem:L2_convI}
Let $\{u_n\}_{n\in \N} \subset \mathcal{A}_n^{F_\lambda}$ be a sequence of piecewise affine functions on the grid $\Omega_n$ satisfying (\ref{eq:surfscaling}). Then there exist (up to zero sets) disjoint Caccioppoli sets $\Omega_0$ and $\Omega_1$  such that
\bea
\label{domconv}
&&\Omega_0\cup\Omega_1=\Omega,\\ \nonumber
&&\dist(\nabla u_n , SO(2)U_0) \rightarrow 0 \mbox{ in } L^2(\Omega_0),\\  \nonumber
&&\dist(\nabla u_n , SO(2)U_1) \rightarrow 0 \mbox{ in } L^2(\Omega_1).
\eea
\end{lem}

\begin{proof}
The proof is based on a comparison argument with a spin-Hamiltonian and is given in Propositions \ref{prop:conver} and \ref{lem:L2_conv} in Appendix A.
\end{proof}
In general the $L^2$ convergence shown in \rf{domconv} can be arbitrarily slow, c.f. Remark \ref{rmk:set1} in Appendix A.

\begin{prop}
\label{prop:rig0}
Let $\{u_n\}_{n\in\N}\subset \mathcal{A}_n^{F_\lambda}$ be a sequence of piecewise affine functions on the grid $\Omega_n$ satisfying the energy bound (\ref{eq:surfscaling}).
Then there exist a subsequence $\{n_j\}_{j\in\N} \subset \N$ and a limiting deformation $ u\in W^{1,\infty}(\Omega, \mR^2)$ such that
\begin{itemize}
\item[(a)] $\nabla u_{n_j}  \rightarrow \nabla u \mbox{ in } L^{2}(\Omega_0)$,
\item[(b)] $\nabla u \in SO(2)U_0 \mbox{ in } \Omega_{0}.$
\end{itemize}
An analogous statement holds in $\Omega_1$.
\end{prop}

In order to show this, we follow an argument of Kinderlehrer, c.f. also \cite{Mue} Theorem 2.4, in which the function $\dist(\cdot, SO(2)U_0)$ is replaced by a lower semi-continuous analogue. 

\begin{proof}
\emph{Step 1: Set-up.}

We set $f(F):= |F U_0^{-1}|^2 - 2\det(F U_0^{-1})$ and remark that $f(F)\geq  0$ and $ f(F)=0 \Leftrightarrow F = \lambda Q U_{0}$ with $\lambda \geq 0$ and $Q\in SO(2)$. 
Moreover, we recall the truncation argument from \cite{FJM05}, which allows to replace $u_n$ by a sequence $v_n\in W^{1,\infty}(\Omega)$ with the following properties:
\begin{align*}
\| \nabla v_n \|_{L^{\infty}(\Omega)} &\leq c_1,\\
\| \nabla u_n - \nabla v_n \|_{L^2(\Omega)} &\leq \int\limits_{\{|\nabla u_n| \geq c_2\}}|\nabla u_n|^2 dx \leq \frac{c}{n}.
\end{align*}
Here $c_1, c_2 > 0$ are universal constants (independent of $n$). The last estimate follows from the energy bounds  (\ref{eq:surfscaling}) and \rf{e:lb}. As a consequence of the comparability of $\nabla u_n$ and $\nabla v_n$, the two functions have the same weak limit $\nabla u$. Therefore, it suffices to prove the strong convergence of $\nabla v_n$ to $\nabla u$.
\\

\emph{Step 2: Convergence of $f(\nabla v_n)$.}
We claim that
\begin{equation}
\label{eq:measconv}
\lim\limits_{n  \rightarrow \infty} \int\limits_{\Omega_0} f(\nabla v_n) dx = 0.
\end{equation}
Indeed, this is a direct consequence of the second estimate in \rf{domconv} in Proposition \ref{lem:L2_convI}: As $L^2$ convergence implies convergence in measure, for any $\epsilon, \delta >0$ given, there exists $N_0= N_0(\epsilon, \delta)$ such that for $\Omega_{\delta}:=\{x\in \Omega_0| \dist(\nabla v_n, SO(2)U_0)< \delta \}$ it holds
\begin{align*}
|\Omega_{\delta}| \geq (1-\epsilon)| \Omega_0| \mbox{ for all } n \geq N_0.
\end{align*}
Thus,
\begin{align*} 
\int\limits_{\Omega_0} f(\nabla v_n) dx & = \int\limits_{\Omega_0   \cap\Omega_{\delta}^c}f(\nabla v_n) dx  + \int\limits_{\Omega_0   \cap \Omega_{\delta}} f(\nabla v_n) dx\\
& \leq  |\Omega_0  \cap \Omega_\delta^{c}|f(c_1) +  C|\Omega|\delta^2 \leq C|\Omega|(\epsilon +  \delta^2).
\end{align*}

\emph{Step 3:} Using the convergence from Step 2, leads to
\begin{equation}
\label{eq:weakFJM}
\begin{split}
0 & = \liminf\limits_{n \rightarrow \infty} \int\limits_{\Omega_{0}}f(\nabla v_n)dx\\
&= \liminf\limits_{n \rightarrow \infty} \left(\int\limits_{\Omega_{0}}|\nabla v_n U_{0}^{-1}|^2dx -2 \int\limits_{\Omega_{0}}\det(\nabla v_n U_{0}^{-1})dx  \right)\\
& \geq \int\limits_{\Omega_{0}}|\nabla u U_{0}^{-1}|^2dx - 2\int\limits_{\Omega_{0}}\det(\nabla u U_{0}^{-1})dx\\
&= \int\limits_{\Omega_{0}}f(\nabla u)dx \geq 0.
\end{split}
\end{equation}
Here the first equality follows from Step 2 above. The fourth estimate is a consequence of the lower semi-continuity of the norm together with the weak continuity of the determinant and the fact that $\nabla v_n  \rightharpoonup \nabla u$ in $L^2(\Omega)$ as $n \rightarrow \infty$. 
Hence, we have equalities everywhere in (\ref{eq:weakFJM}) and therefore,
\begin{align*}
f(\nabla u)=0 \mbox{ and }  \left\| \nabla v_n  \right\|_{L^2(\Omega_{0})}  \rightarrow \left\| \nabla u  \right\|_{L^2(\Omega_{0})} 
\end{align*}
along a subsequence.
As a consequence, along a subsequence $\nabla u_n \rightarrow \nabla u$ in $L^2(\Omega_0)$ and $\nabla u(x) \in SO(2)U_0$ (we remark that $\lambda = 1$ due to the closeness of $\nabla v_n$ to $SO(2)U_0$ which follows from Lemma \ref{lem:L2_conv}).
\end{proof}

With these results at hand, we can now invoke the two-well rigidity result of Dolzmann and Müller \cite{DM95}. This yields a structure result for the sets $\Omega_0$, $\Omega_1$. More precisely, the associated liming deformation, $u$, has to be a simple laminate, which (by virtue of the energy bound (\ref{eq:surfscaling})) implies that $\Omega_{0,1}$ consists of a union of finitely many stripes (and triangles), c.f. Figure \ref{fig:setup1}.

\begin{figure}
\centering
\includegraphics[width=0.75\textwidth]{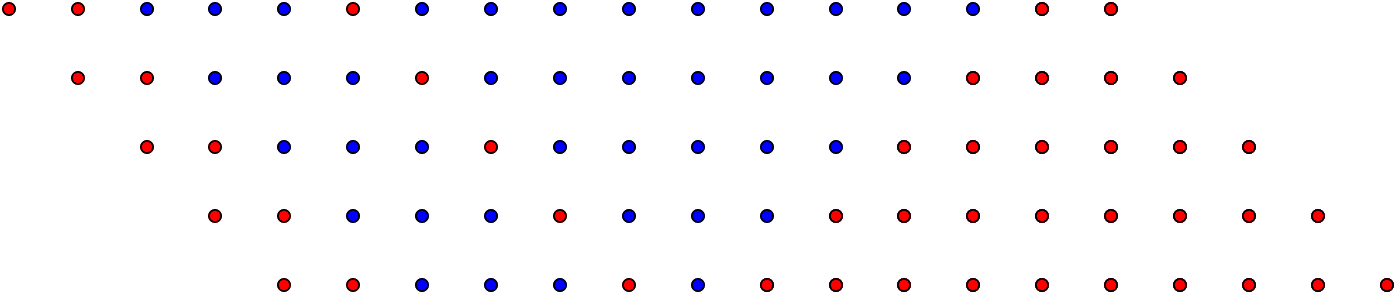}
\caption{A simple laminate depicted in the reference configuration. The figure shows the reference configuration and the domains, in which $ \nabla u \in SO(2)U_0$ (red) and $\nabla u\in SO(2)U_1$ (blue). In particular, the configuration need only locally be a simple laminate. It is also possible that there are intersections of the twinning planes at the boundary.}
\label{fig:setup1}
\end{figure}

\begin{prop}[Rigidity]
\label{prop:rig2}
Let $\{u_n\}_{n\in \N} \subset \mathcal{A}_n^{F_\lambda}$ be a sequence of piecewise affine functions on the grid $\Omega_n$ satisfying the energy bound (\ref{eq:surfscaling}).
Then there exist a subsequence $\{n_j\}_{j\in\N} \subset \N$ and a limiting deformation $u\in W^{1,\infty}(\Omega, \mR^2)$ such that 
\begin{itemize}
\item[(a)] $\nabla u_{n_j}  \rightarrow \nabla u$ in $L^2(\Omega)$, where $\nabla u \in K$ a.e. is a piecewise constant BV function,
\item[(b)] the associated domains $\Omega_0$ and $\Omega_1$ which were defined in Lemma 2.1 consist of a union of finitely many polygonal domains which extend up to the boundary of $\Omega$. The interfaces of the polygonal domains are either given by lines with the normals $(-1,1)$ or $ (1,1)$ (which are determined by the rank-one connections from \rf{eq:ROC}) or by the boundary of $\Omega$.
\item[(c)] there exists a constant $c\in \mR^2$ with
\begin{align*}
u(x_1,x_2) &= F_{\lambda} (x_1,x_2), \ x_{1}+x_2 \leq -4,\\
u(x_1,x_2) &= F_{\lambda} (x_1,x_2) + c, \ x_{1}+x_2 \geq 4,
\end{align*}
for all $|x_2 | \leq 2$.
\end{itemize}
\end{prop}

\begin{proof}
By Proposition  \ref{prop:rig0} there exists a subsequence $n_j\in \N$ such that $\nabla u_{n_j}  \rightarrow \nabla u$ and $\nabla u \in SO(2)U_0$ (and $SO(2)U_1$) in $\Omega_0$ (and $\Omega_1$, respectively). As $\Omega_0$ and $\Omega_1$ have finite perimeter, $\nabla u$ satisfies the conditions of Theorem 5.3 in \cite{DM95} and the limiting deformation $\nabla u$ locally is a laminate. As $\Per(\Omega_1) \leq  C < \infty$, there can only be a finite number of phase transitions between $SO(2)U_0$ and $SO(2)U_1$. The remaining parts of the statements (a), (b) follow from Theorem 5.3 in \cite{DM95}.\\
The boundary conditions which are stated in (c) are a consequence of the $H^1(\Omega)$ convergence of $u$ and the prescribed boundary data (\ref{eq:bc}). The finiteness of $c\in \mR^2$ follows from the energy bound (\ref{eq:surfscaling}).
\end{proof}

The compactness result of Proposition \ref{prop:rig2} is a first important step for the formulation and derivation of the limiting surface energy (c.f. Definition \ref{defi:limiting energy}, (\ref{eq:total_energy1}) and Theorem \ref{prop:surface}) and proof of Theorem \ref{Th1} in Section \ref{sec:Gamma}.

\subsection{Two-well rigidity}
\label{sec:twell}
In this section we derive an analogue of the two-well rigidity result of Conti and Schweizer \cite{CS06}, \cite{CS06a}, \cite{CS06c} in our discrete framework:

\begin{prop}[Two-well rigidity]
\label{prop:twell}
Let $\alpha\in (0,1/8),\ x_0, y_0\in \Omega$ with\\
$|x_0 - y_0|=r > 0$ and $B_{2 \alpha r}(x_0) \cup B_{2 \alpha r}(y_0)\subset \Omega$.
Define $M:= \conv(B_{2\alpha r}(x_0) \cup B_{2\alpha r}(y_0) )$ and suppose that $u\in \mathcal{A}_n$ is piecewise affine on the grid $\Omega_n$ and satisfies
\begin{align}
\label{eq:majority_phase}
& \sum\limits_{(i,j) \in n M} \frac{1}{n^2} (r^{-2} \phi_{u}^{ij}+ r^{-1}|\nabla_n \phi_u^{ij}|) \leq \eta,
\end{align}
where $\phi \in C^{0,1}((\mR^2)^4, \mR),   \ \phi(\cdot) \geq 0,$
\begin{align*}
&(\p_1 u^{i,j}, \p_1 u^{i-1,j}, \p_{2}u^{i,j}, \p_{2}u^{i,j-1}) \mapsto \phi(\p_1 u^{i,j}, \p_1 u^{i-1,j}, \p_{2}u^{i,j}, \p_{2}u^{i,j-1})=:\phi_{u}^{ij}
\end{align*}
denotes a one-well energy density with energy well $SO(2)U_0$ and two-growth behavior, i.e. there exist positive constants $c_1, c_2>0$ such that
\begin{align}
\label{eq:growth_a}
c_1 \dist^2(\nabla u^{ij}, SO(2)U_0) \leq \phi_u^{ij} \leq c_2 \dist^2 (\nabla u^{ij}, SO(2)U_0).
\end{align}
Let $\theta \in (0,1)$. If $\eta = \eta(\theta,\alpha)>0$ is chosen sufficiently small, then there exist a constant $c=c(\alpha, \theta)$ and a subset $U\subset B_{\alpha r}(x_0)\times B_{\alpha r}(y_0)$ with measure\\ $|U|\geq (1-\theta) |B_{\alpha r}(x_0)\times B_{\alpha r}(y_0)|$ such that for all $(x,y)\in U$
\begin{align}
\label{eq:closew}
1 - c\mu \leq \frac{|u(x)-u(y)|}{|U_0(x-y)|} \leq 1 + c\mu, 
\end{align}
where $\mu:=\frac{1}{r^2}\sum\limits_{(i,j)\in n M}\frac{1}{n^2}\dist(\nabla u^{ij},K).$
\end{prop}

\begin{remark}
We recall that the symbol $\nabla_n$ in (\ref{eq:majority_phase}) denotes the discrete gradient as defined in (\ref{eq:discr_grad}).
\end{remark}

\begin{remark}
\label{rmk:rigid}
In the sequel, we call a line segment, $[x,y]$, with endpoints $x,y\in \Omega$ which satisfies (\ref{eq:closew}) \emph{rigid}.  This convention will, for instance, be used in Step 4a in the proof of Proposition \ref{prop:reduc}.
\end{remark}
\begin{remark}
In the proof of Proposition \ref{prop:twell} we closely follow the ideas of Conti and Schweizer \cite{CS06a}, however taking into account our discrete set-up. 
In particular, a priori we do not have enough regularity to apply $C^1$ degree theory. However, as our maps are piecewise affine, the weaker \emph{almost everywhere} results, which follow from degree theory, can be upgraded to hold \emph{everywhere}.\\
\end{remark}
\begin{remark}[Scaling]
We remark that for $r \sim \frac{1}{n}$ the smallness condition (\ref{eq:majority_phase}) turns into a pointwise condition. In the sequel, by scaling we will assume that $r=1$. 
\end{remark}

\begin{proof}
\emph{Step 1: Preliminaries -- Definition and estimates for the bad set.} We note that by (\ref{eq:majority_phase}) $U_0$ is the majority phase in $M$. For 
\bes
0<\tilde{c}_1:= \frac{1}{4}\min\{\bar{c},\bar{c}^2\},
\ees
with $\bar{c}$ as in (\ref{cb}), we consider sets of the form 
$$\Omega_b^{\tilde{c}_1}:=\overline{\{ x\in \Omega   \cap M: \phi_u^{ij} \geq \tilde{c}_1 \}}.$$ 
Here $\overline{\cdot}$ denotes the closure of a set. Since $u$ is a piecewise affine function, it is immediate that $\Omega_b^{\tilde{c}_1}$ consists of a finite union of grid triangles $\D_{ij}^n \subset \Omega_n$: 
\begin{align}
\label{eq:union}
\Omega_b^{\tilde{c}_1} = \bigcup\limits_{ij} \D_{ij}^n \mbox{ and } \partial \Omega_b^{\tilde{c}_1} =  \bigcup\limits_{ij} \partial \D_{ij}^n.
\end{align} 
By the discrete coarea formula (see Appendix C, in particular \rf{eq:coarea1})
\begin{align}
\int\limits_{0}^{\infty}\Per_M(\{(i,j) :\phi_u^{ij}\geq t\})dt \leq C \sum\limits_{(i,j)\in nM} \frac{1}{n^2} |\nabla_n \phi_u^{ij}|,
\end{align}
we infer that for each $\tilde{c}_1>0$ there exists $c_1\in (\tilde{c}_1, 2\tilde{c}_1)$ such that for $\Omega_b:=\Omega_b^{c_1},$ 
it holds
\begin{align*}
\Per_M(\Omega_b) \leq  C\frac{\eta}{\tilde{c}_1}.
\end{align*}
Hence, by the isoperimetric inequality and by the assumption that $U_0$ is the majority phase in our sample $M$, 
we infer that
\begin{align}
\label{eq:vol}
|\Omega_b|\leq C \frac{\eta}{\tilde{c}_1}.
\end{align}
By definition of the ``bad set'' $\Omega_b$, on its complement, i.e. on $M\setminus \Omega_b$, $\nabla u_n$ is $c_1$-close to $SO(2)U_0$. On the boundary however this is not necessarily true. Setting
\begin{align*}
(M\cap \partial \Omega_b)_{\epsilon}:= \{x\in M\setminus \Omega_b: \ \dist(x,\Omega_b)=\epsilon\} \mbox{ for any } \epsilon \in \left(0,\frac{1}{10n}\right),
\end{align*} 
and recalling the continuity of $u$ and the structure of $\partial \Omega_b$, we obtain
\begin{align*}
\mathcal{H}^1(u(M\cap \partial \Omega_b)) = \lim\limits_{\epsilon \rightarrow 0} \mathcal{H}^1(u(M\cap \partial \Omega_b)_{\epsilon}), \
\mathcal{H}^1(M\cap \partial \Omega_b) = \lim\limits_{\epsilon \rightarrow 0} \mathcal{H}^1((M\cap \partial \Omega_b)_{\epsilon}).
\end{align*}
Using these observations and the fact that $u(M  \cap \partial \Omega_b)$ and $u((M\cap \Omega_b)_{\epsilon})$ are rectifiable (as the images of Lipschitz sets under a Lipschitz function), we therefore obtain
\begin{equation}
\label{eq:hausd}
\begin{split}
\mathcal{H}^1(u(M  \cap \partial \Omega_b)) &\leq 2\mathcal{H}^1(u(M\cap\partial \Omega_b)_{\epsilon}) \leq 2\int\limits_{(M   \cap \partial \Omega_b)_{\epsilon}} |\nabla u| d\mathcal{H}^1 \\
&\leq 2(1+c_1) \mathcal{H}^1 ((M   \cap \partial \Omega_b)_{\epsilon}) \leq  4(1 +c_1) \mathcal{H}^1 (M   \cap \partial \Omega_b) \\
&\leq 4C \eta.
\end{split}
\end{equation}
for all sufficiently small $\eps$. Here we  invoked the area formula and used that $\nabla u$ is $c_1$-close to $SO(2)U_0$ in $(M\cap \partial \Omega_b)_{\epsilon}$.
Moreover,
\begin{equation}
\label{eq:det}
\begin{split}
|u(\Omega_b)| & \leq  \int\limits_{\Omega_b}|\det(\nabla u)| dx \leq C\int\limits_{\Omega_b}(1+\dist^2(\nabla u, SO(2)U_0)) dx \\
&\leq C|\Omega_b| + C \sum\limits_{(i,j)\in n M} \frac{1}{n^2} \phi_u^{ij} \leq C \eta,
\end{split}
\end{equation}
where we exploited the two-growth (\ref{eq:growth_a}) and estimated 
\begin{align*}
|\det(\nabla u)| \leq \left\{
\begin{array}{ll}
1, & \mbox{ if } |\nabla u| \leq 20 c_1 / \tilde{c}_1,\\
C\dist(\nabla u, SO(2)U_0)^2, & \mbox{ if } |\nabla u| \geq  20 c_1 / \tilde{c}_1.
\end{array}  \right.
\end{align*}
Thus, this argument yields control on the size of the measure of the bad points in the reference configuration. 
As in \cite{CS06a} this results in a one-sided control of the type (\ref{eq:closew}), which for completeness is presented in Step 2.\\

\emph{Step 2: The upper bound.}
In this step we consider integrals on segments $[x,y]\subset M$ of certain functions $f$ of $\nabla u\in PC(\Omega,\,\mR^2)$.
We will use that for any direction $\frac{x-y}{|x-y|}$ which is not parallel to the grid edges and for almost any point $x$ the restrictions $\nabla u|_{[x,y]}$ and $f(\nabla u)|_{[x,y]}$ are well-defined $L^1([x,y])$ functions.
However, the grid edges constitute a zero set in $B_{\alpha }(x_0) \times B_{\alpha }(y_0)$ and thus we will exclude this set from our consideration.
Additionally we use the convention of saying that certain statements are valid \emph{for most pairs} $(x,y)$ belonging to a certain set. This will mean that the Lebesgue measure
of such pairs $(x,y)$ constitutes at least $(1-\theta)$ total measure of the latter set. Obtaining such a statement will in particular lead to a $\theta$ dependence in the relevant constants. Often even several properties have to be satisfied on a set of size  $(1-\theta)$ of the total measure, which leads to additional losses in the constants.  \\

With these preliminaries and as in \cite{CS06a} we claim that:
\begin{itemize}
\item[(i)] For most pairs $(x,y)\in (B_{\alpha }(x_0) \setminus \Omega_{b}) \times (B_{\alpha }(y_0) \setminus \Omega_b)$ it holds
\begin{align*}
[x,y]  \cap \Omega_b = \emptyset.
\end{align*}
\item[(ii)] For most pairs $(x,y)\in (B_{\alpha }(x_0) \setminus \Omega_{b}) \times (B_{\alpha }(y_0) \setminus \Omega_b)$ the restriction $\nabla u|_{[x,y]}$ is a well-defined $L^1([x,y])$ function. Moreover, there exists a constant $c=c(\theta,\alpha)>0$ such that for most pairs $(x,y)\in (B_{\alpha }(x_0) \setminus \Omega_{b}) \times (B_{\alpha }(y_0) \setminus \Omega_b$
\begin{align}
\label{eq:good_est}
\int\limits_{[x,y]}\dist(\nabla u, SO(2)U_0) d\mathcal{H}^1 \leq c \mu,
\end{align}
where $[x,y]$ denotes the line segment connecting $x$ and $y$.
\item[(iii)] There exists a constant $c=c(\theta,\alpha)>0$ such that following estimate is true
\begin{align}
\lb{URE}
|u(x)-u(y)| \leq |U_0(x-y)| + c \mu. 
\end{align}
\end{itemize}
We begin by discussing properties (i) and (ii). To this end we first observe that, by virtue of (\ref{eq:vol}), 
\begin{align*}
|B_{\alpha }(x_0) \setminus \Omega_{b}| \geq c \alpha^2, \ |B_{\alpha }(y_0) \setminus \Omega_{b}| \geq c \alpha^2,
\end{align*} 
for all sufficiently small $\eta$ (chosen sufficiently small in dependence of $\theta$ in order to guarantee that a sufficiently large volume of ``good points'' is left). As in \cite{CS06a} we now consider the function $f(x):= \dist(\nabla u(x), SO(2)U_0 \cup SO(2)U_1)\chi_{M}$. 
By averaging we show that for almost all $(x,y)\in B_{\alpha}(x_0)\times B_{\alpha}(y_0)$
\begin{align}
\label{eq:restr}
\int\limits_{[x,y]}f d\mathcal{H}^1  <\infty,
\end{align} 
i.e. $f|_{[x,y]}\in L^1([x,y])$.
Considering $\nu\in S^1$ and the line $t \mapsto x + t \nu$, the vector $\nu_0 := \frac{(y_0 - x_0)}{|y_0-x_0|}$ and noting that only those lines with $|\nu-\nu_0|\leq 3 \alpha$ intersect $B_{\alpha r}(y_0)\setminus \Omega_b$, yields
\begin{equation}
\label{eq:average_1}
\begin{split}
\int\limits_{(B_{\alpha }(x_0) \setminus \Omega_{b}) \times (B_{\alpha }(y_0) \setminus \Omega_b)} \int\limits_{[x,y]}f d\mathcal{H}^1 
& \leq c \alpha \int\limits_{|\nu- \nu_0|\leq 3 \alpha} \int\limits_{B_{\alpha }(x_0) \setminus \Omega_{b}} \int\limits_{\mathbb{R}} f(x+t \nu) dt dx d\nu\\
& \leq C \alpha^2 \int\limits_{M} f dx \leq C \alpha^3 \mu.
\end{split}
\end{equation}
This shows (\ref{eq:restr}). Keeping this in mind, we now prove (i). As in \cite{CS06a} the result follows from a projection argument. In fact we claim that for any $\nu\in S^1$ the set 
\begin{align*}
A_{\nu}:= \{ x\in  B_1 \setminus \Omega_{b}:\ t \mapsto x + t\nu \mbox{ intersects } \partial \Omega_b\}
\end{align*}
satisfies 
\begin{align*}
|A_{\nu}| \leq 2 \alpha \mathcal{H}^1(\partial \Omega_b) \leq c \alpha \eta.
\end{align*}
Indeed, writing $x=x_0 + s\nu+ r \nu^{\perp}$ for each $x\in A_{\nu}$, we have that $x_0 + r \nu^{\perp}$ is contained in the orthogonal projection of $\partial \Omega_b$ onto the line through $x_0$ parallel to $\nu^{\perp}$. This yields the claim of (i). Choosing $c=c(\theta,\alpha)> \frac{C \alpha^3}{\theta}$, where $C$ denotes the constant from (\ref{eq:average_1}), then also yields (ii). Here we exploited that by (i), $f= \dist(\nabla u,SO(2) U_0 )$ along most of the line segments $[x,y]$ and that by our choice of $c=c(\theta,\alpha)$, (\ref{eq:average_1}) implies that (\ref{eq:good_est}) holds on a set of measure of size at least $(1-\theta)|B_{\alpha}(x_0)\times B_{\alpha}(y_0)|$.\\

Last but not least we present the argument for (iii). By (ii) we have
\begin{align*}
|u(x)-u(y)| &\leq \int\limits_{[x,y]}|\nabla_{\tau} u| d\mathcal{H}^1 \\
& \leq |U_0(x-y)| + c\int\limits_{[x,y]} \dist(\nabla u, SO(2)U_0) d\mathcal{H}^1\\
& \leq |U_0(x-y)| + c \mu,
\end{align*}
where $\nabla_{\tau} u$ denotes the tangential derivative.\\

\emph{Step 3: The lower bound.}
As in \cite{CS06a} the lower bound is the most delicate part of the proof, since our (piecewise affine) deformation is not necessarily invertible. In addition to this key difficulty, which was overcome by Conti and Schweizer \cite{CS06a} by showing that the deformation is invertible on ``a sufficiently large set'', we have to deal with the fact that our deformation is \emph{not a priori regular}. Since we are working with a piecewise affine interpolation of our discrete data, we cannot immediately invoke the degree theory arguments of Conti and Schweizer (which a priori only apply to $C^1$ functions). Instead we use an integrated formula for the degree.\\ 

Again we argue in five steps:
\begin{itemize}
\item[(i)] (Approximate injectivity) 
Let $\alpha''= 3\alpha/2$ and $M'':= \inte \conv(B_{\alpha'' }(x_0) \cup B_{\alpha'' }(y_0))$. 
Then there exists $\alpha'\in (7 \alpha/4, 2 \alpha)$ and an affine deformation $A:M  \rightarrow \mathbb{R}^2$ with gradient in $SO(2)U_0$, such that for all $z\in A(M'') \setminus u(\Omega_b)$ there exists a unique preimage $x\in M'   \cap u^{-1}(z)$. Here $M':= \inte \conv(B_{\alpha' }(x_0) \cup B_{\alpha' }(y_0))$.  
\item[(ii)] Let $A$ be the mapping from (i). Then for the set
\begin{align*}
\Omega_{bi}:= \{ x \in M': \ u(x) \in u(\Omega_b)  \cap A(M'') \}
\end{align*}
we have $|\Omega_{bi}|\leq c \eta$. 
\item[(iii)] We claim that for most choices of $(x,y) \in (B_{\alpha''}(x_0) \setminus \Omega_b) \times (B_{\alpha'' }(y_0) \setminus \Omega_b)$
\begin{itemize}
\item[(a)]
$ [u(x), u(y)]   \cap u(\partial \Omega_{b}) = \emptyset,$
\item[(b)] $ [u(x), u(y)] \subset u(M' \setminus \Omega_{b}), $
\item[(c)] there exists a piecewise affine curve $\gamma_{xy}(t):[0,1] \rightarrow [u(x),u(y)]$, such that $u \circ \gamma_{xy}$ is a monotonic parametrization of the segment $[u(x),u(y)]$.
\end{itemize}
\item[(iv)] Let $\gamma_{xy}$ be a piecewise affine curve such that $u \circ \gamma_{xy}$ is a monotonic parametrization of the segment $[u(x),u(y)]$. Then for most pairs $(x,y)\in (B_{\alpha''}(x_0) \setminus \Omega_b) \times (B_{\alpha''}(y_0) \setminus \Omega_b)$ we have that $\dist(\nabla u, SO(2)U_0)|_{\gamma_{xy}}$ is an $L^{1}(\gamma_{xy})$ function. Furthermore, there exists a constant $c=c(\theta,\alpha)>0$ such that for most pairs $(x,y)$
\begin{align*}
\int\limits_{\gamma_{xy}}\dist(\nabla u, SO(2)U_0) d\mathcal{H}^1 \leq c \mu.
\end{align*}
\item[(v)] There exists a constant $c=c(\theta,\alpha)>0$ such that for most pairs $(x,y)\in (B_{\alpha''}(x_0) \setminus \Omega_b) \times (B_{\alpha'' }(y_0) \setminus \Omega_b)$
\begin{align*}
|u(x)-u(y)| \geq |U_0(x-y)|-c\mu.
\end{align*}
\end{itemize}
The crucial point in the argument consists of the derivation of (i). As soon as this step is established, the remaining argument follows along the lines of~\cite{CS06a}.\\

In order to deduce (i), we argue in various steps. Firstly, following \cite{CS06a} and using (\ref{eq:growth_a}), we observe that 
\begin{align*}
\int\limits_{M }\dist^2(\nabla u, SO(2)U_0) dx &\leq  C \sum\limits_{(i,j) \in n M} \frac{1}{n^2} \phi_u^{ij}  \leq C \eta.
\end{align*}
Hence, the Friesecke, James, Müller rigidity theorem \cite{FJM05} implies that there exists $Q\in SO(2)$ such that
\begin{align*}
\int\limits_{M}|\nabla u - QU_0|^2 dx  \leq C \eta.
\end{align*}
Combined with Poincar\'e's inequality and setting $A=QU_0 x +b$, this leads to
\begin{align*}
\int\limits_{M}|\nabla u - \nabla A|^2 + |u- A|^2 dx  \leq C \eta.
\end{align*}
Thus, there exists a radius $\alpha'\in (7 \alpha/4, 2 \alpha)$ such that for $M':= \inte\conv(B_{\alpha' }(x_0) \cup B_{\alpha' }(y_0))$ we have
\begin{align*}
\int\limits_{\partial M'}|\nabla u - \nabla A|^2 + |u- A|^2 dx  \leq C \eta.
\end{align*}
By the embedding of $W^{1,2}(\partial M')  \rightarrow L^{\infty}(\partial M')$ we hence infer
\begin{align*}
\| u- A\|_{L^{\infty}(\partial M')} \leq c \eta^{\frac{1}{2}}.
\end{align*}
As in \cite{CS06a}, this allows us to conclude that for $\alpha'':= 3 \alpha/2$, $M'':=\inte\conv(B_{\alpha'' }(x_0) \cup B_{\alpha'' }(y_0))$ and a sufficiently small choice of $\eta>0$, the degree of 
\begin{align*}
v_t(x):=tu(x) + (1-t)A(x)
\end{align*}
is well defined in $A(M'')$ (as $v_t(\partial M')\cap A(M'') = \emptyset$). By homotopy invariance
\begin{align}
\label{eq:deg}
\deg(u,M',z)= 1 \mbox{ for all } z\in A(M'').
\end{align}
We note that $M''$ is open and as $A$ is affine, this is also true for $A(M'')$. Moreover, $\overline{B_{\alpha }(x_0) \cup B_{\alpha }(y_0)}$ is contained in $M''$.
Hence, by the change of variables formula in terms of the degree and the multiplicity function, we have
\begin{equation}
\label{eq:deg1}
\begin{split}
\int\limits_{\mathbb{R}^2}v(z)d(u, M', z) dz &= \int\limits_{M'}v\circ u(x) \det(\nabla u(x)) dx,\\
\int\limits_{\mathbb{R}^2}v(z)N(u, M', z) dz &= \int\limits_{M'}v\circ u(x) |\det(\nabla u(x))| dx,
\end{split}
\end{equation}
for all $v\in L^{\infty}(\mathbb{R}^2)$.
Thus, recalling that (by the non-interpenetration condition $u\in\mathcal{A}_n$) $\det(\nabla u)>0$ almost everywhere in $M$  and choosing functions $v\in L^{\infty}(\mR^2)$ with $\mbox{supp}(v)\subset A(M'')\setminus u(\Omega_{b})$, we obtain
\begin{equation}
\label{eq:deg2}
\begin{split}
\int \limits_{\mathbb{R}^2}v(z) dz & = \int\limits_{\mathbb{R}^2}v(z)d(u, M', z) dz = \int\limits_{M'}v\circ u(x) \det(\nabla u(x)) dx\\
&=\int\limits_{\mathbb{R}^2}v(z)N(u, M', z) dz.
\end{split}
\end{equation}
Therefore, 
\begin{align*}
N(u,M',z)=1 \mbox{ for almost all } z\in A(M'')\setminus u(\Omega_b). 
\end{align*}
This is the desired uniqueness of the preimage of $z$ under $u$ for \emph{almost all} $z\in A(M'')\setminus u(\Omega_b)$. 
We now argue, that this uniqueness statement can be extended to \emph{all} $z\in A(M'')\setminus u(\Omega_b)$. Indeed, this is a consequence of the definition of $\Omega_b$ (which is a closed set), the fact that our function $u$ is piecewise affine with a locally invertible gradient and the implicit function theorem.
We argue by contradiction: Assume that there existed $z\in A(M'')\setminus u(\Omega_b)$ such that it had two preimages $y_1, y_2 \in M'$. By definition, $y_1,y_2 \notin \Omega_b$. We argue that this already implies the existence of a set of positive measure in $A(M'')\setminus u(\Omega_b)$ on which the injectivity of $u$ is violated, which can not be the case by our almost everywhere invertibility result from above.
To this end, we distinguish three different cases:
\begin{itemize}
\item[(a)] We first assume that there exist grid triangles (which without loss of generality we take as ``+'' triangles) $\Delta_{ij}^{n,+}, \Delta_{kl}^{n,+} \subset M'$ such that $y_1 \in \inte(\D_{ij}^{n,+})$, $y_2\in \inte(\D_{kl}^{n,+})$. By definition of $\Omega_b$ as the \emph{closure} of the ``bad set'' and by the piecewise affine definition of $u$, this yields that also $\Delta_{ij}^{n,+}, \Delta_{kl}^{n,+} \subset M' \setminus \Omega_b$. 
Since $\nabla u$ is an affine function on each of the triangles and as $\det(\nabla u)\neq 0$ on these, the implicit function theorem immediately yields that there exist whole neighborhoods $U_1,U_2$ of $y_1,y_2$ such that $u(U_1)=u(U_2)$. In particular, using the openness of $A(M'')$, we would obtain a set $V\subset A(M'')\setminus u(\Omega_b)$ of non-zero measure such that for all $z\in V$ the preimage under $u$ is not unique. This yields a contradiction.
\item[(b)] We now suppose that $y_1 \in \D_{ij}^{n,+} \cap M'$, $y_2\in \D_{kl}^{n,+} \cap M'$ are such that at least one of the points $y_1$ or $y_2$ lies on $\partial \D_{ij}^{n,+}$ or $\partial \D_{kl}^{n,+}$, but neither of them is a vertex of the underlying grid. In this case we cannot directly argue via the implicit function theorem as our function is only Lipschitz regular. However, assuming that for instance the point $y_1$ lies at the interface between two grid triangles $\D_{rs}^{n,\pm}$ and $\D_{ij}^{n,\pm}$ (the superscript $\pm$ which is used here indicates that we do not make an assumption in which of the two grid triangles this may be the case), we can invoke the implicit function theorem in each of the triangles (using the invertibility of the gradient on each of the triangles) to obtain two neighborhoods $\tilde{U}_1  \subset \D_{rs}^{n,\pm}\cap M', \tilde{U}_2 \subset \D_{ij}^{n,\pm}\cap M'$ (for which we use the openness of $M'$) whose closures intersect along a line containing $y_1$. In order to avoid a contradiction to the almost everywhere invertibility already at this point, $u$ has to be one-to-one on $U_1:=\overline{\tilde{U}_1} \cup \overline{\tilde{U}}_2$. In particular, $u$ maps $U_1$ onto a full neighborhood of $z:=u(y_1)$.
But arguing similarly for $y_2$ and using the openness of $A(M'')$ we again obtain a set $V\subset A(M'')\setminus u(\Omega_b)$ of non-zero measure such that $u$ does not have a unique preimage.
\item[(c)] Finally, we have to cover the case where at least one of the points $y_1$ or $y_2$ is a vertex of (at least) one of the triangles $\D_{ij}^{n,+}$ or $\D_{kl}^{n,+}$. However, similarly as above we can again construct full neighborhoods of $y_1$ and $y_2$ on which injectivity is violated, which yields a contradiction.
\end{itemize}
Hence, we conclude that, indeed, for all $z\in A(M'')\setminus u(\Omega_b)$ there is a unique preimage $x\in M'$.\\

We now proceed with the remaining points (ii)-(v). Thanks to (i), this follows along the lines of the argument of Conti and Schweizer \cite{CS06a}.\\

Recalling the definition of $\Omega_{bi}$ from (ii), we estimate
\begin{align*}
|\Omega_{bi}| &\leq |M'   \cap \Omega_b| + \int\limits_{M' \setminus \Omega_b} \chi_{\Omega_{bi}} dx\\
& \leq |\Omega_{b}| + 2 \int\limits_{u(\Omega_b)  \cap A(M'')}\#(u^{-1}(z)  \cap \Omega' \setminus \Omega_{b}) dz\\
& \leq C\eta + 2 \int\limits_{u(\Omega_b)  \cap A(M'')}\#(u^{-1}(z)  \cap \Omega' \setminus \Omega_{b})dz.
\end{align*}
Here we used the closeness of $\nabla u$ to $SO(2)U_0$ away from $\Omega_b$.
Due (\ref{eq:deg1}), for almost every $z\in \Omega_b   \cap A(M'')$
\begin{align*}
\#(u^{-1}(z)  \cap M' \setminus \Omega_b) \leq 1 + \#(u^{-1}(z)  \cap M'   \cap \Omega_b).
\end{align*}
Moreover, the coarea formula, (\ref{eq:det}) and the estimate for the determinant in terms of $\phi^{ij}_{u}$ give
\begin{align*}
\int\limits_{u(\Omega_b)}\#(u^{-1}(z)  \cap \Omega_{b})dx
& = \int\limits_{\Omega_b} |\det(\nabla u)| dx  \leq C \eta.
\end{align*}
 This yields the claim. 
\\

By (ii) for most choices of $(x,y)\in M' \times M'$ we have that $u(x), u(y) \not\in u(\Omega_b)$. By a similar projection argument as in Step 2, we hence obtain that $[u(x),u(y)]  \cap u(\partial \Omega_{b}) = \emptyset$ for most pairs $(x,y)$. Indeed, to this end, for $z\in\{x_0,y_0\}$ we consider the sets 
\begin{align*}
\tilde{B}_{\alpha' }(z):= \{ x\in B_{\alpha' }(z)\setminus(\Omega_b  \cup \Omega_{bi}) : |u(x)-A(x)|\leq \frac{\alpha}{4} \}.
\end{align*}
Then, by the size estimates for $\Omega_b$ and $\Omega_{bi}$ (i.e. by equation (\ref{eq:vol}) and claim (ii) in Step 3 from above), this still covers nearly the original volume of $B_{\alpha' }(z)$, $i\in\{1,2\}$, $z\in\{x_0,y_0\}$ if $\eta$ (as a function of $\theta$) is chosen sufficiently small. Reasoning by a projection argument once more, we consider the lines $t\mapsto \xi + t \nu$ for $\nu$ such that $|\nu-(\nabla A) \nu_0|\leq 3 \alpha$. Considering the set
\begin{align*}
\tilde{A}_{\nu}:= \{\xi\in u(\tilde{B}_{\alpha' }(x_0)): x+t\nu \mbox{ intersects } u(\partial \Omega_{b})\},
\end{align*}
and using estimate (\ref{eq:hausd}), we obtain the claim of (iii) (a).
Part (b) now immediately follows from part (a) by noting that for all $x\in \tilde{B}_{\alpha }(z)$, $z\in\{x_0, y_0\}$, by definition $u(x)\in u(B_{\alpha'}(x_0) \setminus  (\Omega_b  \cup \Omega_{bi}))$, and by using (a). Moreover, the claim of (c) follows directly from (b) and the invertibility of $u$ along for all $z\in A(M'')\setminus u(\Omega_b)$. 

The argument for (iv) follows along the lines of \cite{CS06a} and is very similar to point (i) in Step 2.
We however have to establish that for most pairs $(x,y)$ the restrictions of $\nabla u$ onto the line segments $\gamma_{xy}$ are well-defined as $L^1(\gamma_{xy})$ functions.
Indeed, the well-definedness of the restriction follows from the claim that for most pairs $(x,y)$ the piecewise affine curve $\gamma_{xy}$ does not contain line segments of $G^n(\Omega)$. To this end, we observe that $u(G^n):= \bigcup\limits_{[z_1,z_2]\in G^n(\Omega)} [u(z_1),u(z_2)]$ forms a zero set in $u(\tilde{B}_{\alpha' }(x_0)) \times u(\tilde{B}_{\alpha' }(y_0)) $. Further, $\gamma_{xy}$ can only contain a line segment in $G^n$ if there exists a line segment $[u(x_1), u(x_2)] \subset [u(x),u(y)]$ with $[u(x_1), u(x_2)] \subset u(G^n)$. But $u(G^n)$ is a zero set, thus for almost all pairs $(x,y)$ this does not happen.

Last but not least, we observe that by the monotonicity of $u$ along the curve $\gamma_{xy}$ and the fact that for most curves $\gamma_{xy}$ the restriction $\nabla u|_{\gamma_{xy}}$  is a well-defined $L^1(\gamma_{xy})$ function, we have
\begin{align*}
|u(x)- u(y)| &= \int\limits_{\gamma_{xy}}|\nabla_{\tau}u|d \mathcal{H}^1  \geq \mathcal{H}^1(U_0 \gamma_{xy}) - \int\limits_{\gamma_{xy}}\dist(\nabla u, SO(2)U_0) d\mathcal{H}^1\\
& \geq |U_0(x-y)| - c \mu.
\end{align*}
As before, this requirement that this holds for most pairs $(x,y)\in B_{\alpha}(x_0)\times B_{\alpha}(y_0)$ leads to a $\theta$ dependence of the constant $c=c(\alpha,\theta)>0$.
This concludes the proof of Proposition \ref{prop:twell}.
\end{proof}

\section{Surface Energies}
\label{surface}
In this section we investigate the emergence and form of surface energies: After introducing the limiting surface energies in Section  \ref{sec:setup} we deduce some fundamental properties of these in Section  \ref{sec:prop_energy} and finally prove the desired $\Gamma$-limit in Section  \ref{sec:Gamma}.
\subsection{Setting}
\label{sec:setup}
In the context of deformations on $\Omega$ which are in the surface energy scaling regime (\ref{eq:surfscaling}), we define 
\begin{align*}
H_n^1(u_n) := n H_n(u_n).
\end{align*}
Due to its scaling, we interpret it as a surface energy. Before formulating our main result on the limiting structure of $H_n^{1}(u_n)$ as $n  \rightarrow \infty$, we introduce the central objects of this section. We start by defining the \emph{limiting profiles}:
\begin{defi}[Limiting profiles]
\label{defi:profile}
Let $V \in \{U_0, Q U_1 \}$ and $V_1,V_2 \in SO(2)U_0   \cup SO(2) U_1$ be two rank-one connected matrices. Let $F_{\lambda}:= \lambda U_0 + (1-\lambda) Q U_1$ be as in (\ref{eq:Fl}). Then we define the \emph{limiting profiles} as
\begin{align*}
v_{F_{\lambda}, V}(x)&:=  \left\{ \begin{array}{ll}
F_{\lambda} x &\mbox{ for } x\cdot \begin{pmatrix}1,1\end{pmatrix} \leq 0,\\
V x & \mbox{ for } x \cdot \begin{pmatrix}1,1\end{pmatrix} > 0,
\end{array}  \right.\\
v_{V, F_{\lambda}}(x)&:=  \left\{ \begin{array}{ll}
F_{\lambda} x &\mbox{ for } x\cdot \begin{pmatrix}1,1\end{pmatrix} \geq 0,\\
V x & \mbox{ for } x \cdot \begin{pmatrix}1,1\end{pmatrix} < 0,
\end{array}  \right.\\
v_{V_1, V_2}^{\pm}(x)&:=  \left\{ \begin{array}{ll}
V_1 x &\mbox{ for } x\cdot \begin{pmatrix}\pm1,1\end{pmatrix} \leq 0,\\
V_2 x & \mbox{ for } x \cdot \begin{pmatrix}\pm 1,1\end{pmatrix} > 0.
\end{array}  \right.
\end{align*}
\end{defi}

With this at hand, we introduce the following boundary and internal layer energies. 

\begin{defi}[Boundary and internal layer energies]
\label{defi:boundarylayers}
Let $V \in \{U_0, Q U_1 \}$ and $V_1,V_2 \in K$ be two rank-one connected matrices. Let $F_{\lambda}$ be as in (\ref{eq:Fl}). 
Then we define the \emph{left boundary layer energy}, $B_{+}(F_\lambda,V)$, the \emph{internal layer energies}, $C_{\pm}(V_1,V_2)$, and the \emph{right boundary layer energy}, $B_{-}(F_\lambda,V)$ as
\begin{equation}
\label{eq:layer_energy1}
\begin{split}
B_+(F_{\lambda},V) &:= \inf\{ \liminf\limits_{n  \rightarrow \infty} \sum\limits_{(i,j)\in n\Omega} \frac{1}{n} h_{v_n}^{i,\,j}:
\,v_n\in\mathcal{A}_{n\Omega}, \  v_n  \rightarrow v_{F_{\lambda},V} \mbox{ in } L^1(\Omega),\\
& \quad \quad \quad v_n(i/n,j/n)= F_{\lambda}(i/n,j/n) \mbox{ for } i+j\leq -4n\},\\
B_-(V,F_{\lambda}) &:=  \inf \{ \liminf\limits_{n  \rightarrow \infty}  \sum\limits_{(i,j)\in n\Omega} \frac{1}{n}  h_{v_n}^{i,\,j}:
\, v_n\in\mathcal{A}_{n\Omega}, \  v_n  \rightarrow v_{V,F_{\lambda}}\mbox{ in } L^1(\Omega)\\
& \quad \quad \quad v_n(i/n,j/n)= F_{\lambda}(i/n,j/n) \mbox{ for } i+j\geq 4n\},\\
C_{\pm}(V_1, V_2) &:=  \inf \{  \liminf\limits_{n  \rightarrow \infty} \sum\limits_{(i,j)\in n\Omega_{4,1}^{\pm}} \frac{1}{n}  h_{v_n}^{i,\,j}:
\, v_n\in\mathcal{A}_{n\Omega_{4,1}^\pm}, \  v_n  \rightarrow v_{V_1,V_2}^{\pm} \mbox{ in } L^1(\Omega_{4,1}^{\pm})\},
\end{split}
\end{equation}
where the domains $\Omega, \Omega_{l,d}^{\pm}$ and the sets $\mathcal{A}_{n\Omega_{4,1}^\pm}$ are as in Definition \ref{defi:setup}.
\end{defi}

We show that in the sense of $\Gamma$-limits we can identify the energy $H^1_n$ with an energy which is concentrated on the jump surfaces of a limiting configuration $u_0$. We recall that the limiting deformations which arise from passing to the limit $n\rightarrow \infty$ of discrete deformations in the surface energy scaling regime (\ref{eq:surfscaling}), are \emph{rigid} (c.f. Remark \ref{rmk:rigid}). More precisely, they satisfy the structure result of Proposition \ref{prop:rig2} and are hence locally simple laminates. Using this, we give the following definitions:

\begin{defi}[Limiting energy]
\label{defi:limiting energy}
Let $u_0$ be a piecewise affine function with gradient $\nabla u_0 \in SO(2)U_0  \cup SO(2)U_1$. Suppose that it satisfies the boundary conditions (\ref{eq:bc}) and that it has finitely many jump interfaces which pass through the points $(x_l,0)$, with $l\in\{0,\dots,L-1\}$ for $L\in \N$. Let the boundary and internal layers be as in Definition  \ref{defi:boundarylayers}.
Then we set
\begin{equation}
\label{eq:total_energy1}
\begin{split}
&\bar{E}_{surf}(u_0) \\
&:= B_+(F_{\lambda}, \nabla u_0(x_{0}-,0))  + \sum\limits_{i=1}^{L-1} C_{\pm}(\nabla u_0(x_{i}-,0), \nabla u_0(x_{i+1}-,0))\\
& \quad + B_-(\nabla u_0(x_{(L-1)}-,0), F_{\lambda}),\\
&=: \int\limits_{J_{\nabla u_0}} \bar{C}(\nabla u_0(x-,0), \nabla u_0(x+,0)) d\mathcal{H}^1,
\end{split}
\end{equation}
where $u_0(x-):= \lim\limits_{\substack{y  \rightarrow x,\\ y_1 \leq x_1}}u_0(y)$, and, depending on the position and orientation of the jump plane and the values of $\nabla u_0$ at $x\in J_{\nabla u_0}$, the density $\bar{C}(\cdot, \cdot)$ satisfies $\sqrt{2}\bar{C}(\cdot)\in \{B_{+}(\cdot, \cdot), B_{-}(\cdot, \cdot), C_{\pm}(\cdot, \cdot)\}$.
\end{defi}

With these notions at hand, we can finally formulate our main result regarding surface energies:

\begin{theorem}[Surface energies]
\label{prop:surface}
With respect to the $L^{1}$ topology we have that
\begin{align*}
H_n^1 \stackrel{\Gamma}{ \rightharpoonup} E_{surf},
\end{align*}
where 
\begin{align*}
E_{surf}(u):=\left\{ \begin{array}{ll}
\bar{E}_{surf}(u), \mbox{ if } u\in W^{1,\infty}_0(\Omega)+F_{\lambda}x + c \mbox{ for some } c \in \mR^2, \\
 \ \ \ \ \ \ \  \mbox{ with } \nabla
u\in\{U_0,QU_1\} \mbox{ and } \nabla u \in BV;\\ 
\infty, \ \ \mbox{else}.
\end{array}  \right.
\end{align*}
Here $J_{\nabla u}$ denotes the \emph{jump set} of $\nabla u$ and $u(x-):= \lim\limits_{\substack{y  \rightarrow x,\\ y_1 \leq x_1}} u(y)$.
\end{theorem}

Similarly as in \cite{BC07}, \cite{CS06} and \cite{KLR14} the proof of Theorem \ref{prop:surface} is based on a combination of the rigidity result of Proposition  \ref{prop:twell} together with a special cutting procedure. Heading for this, we begin by recalling some properties of the energy in Section  \ref{sec:prop_energy} and then carry out the proof of the $\Gamma$-limit in Section  \ref{sec:Gamma}.

\subsection{Properties of the energy and auxiliary results}
\label{sec:prop_energy}

Before addressing the proof of Theorem \ref{prop:surface}, we discuss central properties of the energy and derive auxiliary results. We begin by considering the energy densities from (\ref{eq:layer_energy1}). For notational convenience we limit ourselves to the case of interior layer energies, the situation for boundary energies is analogous. We start by introducing restricted versions of the internal layer energies from Definition \ref{defi:boundarylayers}:

\begin{defi}
\label{eq:intfaceen}
Let $m_1,m_2 \in \mR\setminus \{0\}$ and let $V_1,V_2 \in K$ be two rank-one connected matrices. Then we set
\begin{equation}
\label{eq:layer_energy2}
\begin{split}
C_{\pm}(V_1, V_2, m_1, m_2)&:= \inf \{  \liminf\limits_{n  \rightarrow \infty}\sum\limits_{\substack{(i,j)\in n\Omega_{m_1,m_2}^{\pm}}}  \frac{1}{n} h_{v_{n}}^{i,j}:\, v_{n}\in\mathcal{A}_{n\Omega_{m_1,m_2}^{\pm}},\\
&\quad  \quad \quad \quad v_{n}  \rightarrow v_{V_1,V_2}^{\pm} \mbox{ in } L^1(\Omega^{\pm}_{m_1,m_2})\}.
\end{split}
\end{equation}
\end{defi}

Analogous definitions hold for the boundary layer energies. We claim that the energies $C_\pm (V_1,V_2,m_1,m_2)$ do not depend on the dimension $m_1$ and are linear in the $m_2$ dimension.

\begin{lem}
\label{lem:indep}
Let $m_1,m_2 \in \mR \setminus \{0\}$ and let either $V_1 = U_0$, $V_2 = Q U_1$ or $V_1 = Q U_1$, $V_2 = U_0$, with the matrix $Q$ from \rf{eq:ROC}.
Let $C_{\pm}(V_1, V_2, m_1, m_2)$ be as in Definition \ref{defi:boundarylayers}. Then there exist constants $C_\pm$ 
depending only on the normals $(\pm 1,1)$ such that
 \begin{equation}
\label{eq:layer_energy3}
C_{\pm}(V_1, V_2, m_1, m_2) = C_{\pm}(V_1,V_2,1,1)m_2 = C_\pm m_2.
\end{equation}
\end{lem}

\begin{proof}
The proof follows from averaging and scaling as in \cite{CS06}, Lemma 3.2 (due to our discrete set-up we however make small errors for each fixed $n\in \N$, these vanish in the limit $n\rightarrow \infty$). We only present the argument for $C_{-}(V_1,V_2,m_1,m_2)$ (the one for $C_{+}(V_1,V_2,m_1,m_2)$ is analogous) and only argue that $C_{-}(V_1,V_2,m_1,m_2)$ is independent of $m_1$, the other dependences being more direct. We begin by noticing that, by definition, $C_-(V_1,V_2,m_1,m_2)$ is an increasing function in $m_1$. We claim that moreover
\begin{align*}
C_-(V_1, V_2, \alpha m_1, \alpha m_2)  
= \alpha C_-(V_1,V_2, m_1, m_2).
\end{align*}
Indeed, assuming that $n$ is sufficiently large and setting $\bar{n}:= [\alpha n]$, we have
\begin{align*}
&\sum\limits_{\substack{(i,j)\in n\Omega_{\alpha m_1,\alpha m_2}^-
}}  \frac{ 1}{n} h_{ v_{n}}^{ij} =
 \alpha \sum\limits_{\substack{(i,j)\in n\Omega_{m_1\alpha,m_2\alpha}^-
}}  \frac{ 1}{(\alpha n)} h_{v_{n}}^{ij} \\
& =
\alpha \sum\limits_{\substack{(i,j)\in \bar{n}\Omega_{m_1,m_2}^-
}}  \frac{ 1}{\bar{n}}h_{u_{\bar{n}}}^{ij} + O(1/\bar{n}),
\end{align*}
where $u_{\bar{n}}(i,j) := \frac{1}{\alpha}v_{\bar{n}}(\alpha i, \alpha j)$. Fixing $\alpha$ and taking the $\liminf$ as $n\rightarrow \infty$ yields the claim.\\ 
Next, we show that
\begin{align*}
C_-(V_1, V_2, m_1, \frac{m_2}{m}) \leq \frac{1}{m} C_-(V_1,V_2,m_1,m_2).
\end{align*}
This follows from averaging and translating. More precisely, we have
\begin{align*}
\sum\limits_{k=0}^{m-1}\left( \sum\limits_{\substack{(i,j)\in n\Omega_{m_1,m_2/m}^-
}}  \frac{ 1}{n} h_{v_{n}^{k}}^{ij} \right) = 
\sum\limits_{\substack{(i,j)\in n \Omega_{m_1,m_2}^-
}}  \frac{ 1}{n} h_{v_{n}^{ij}} + O(m/n),
\end{align*}
where 
\begin{align*}
( v_{ n}^{k})^{i j}:= v_n(i/n + k m/n, j/n + k m/n).
\end{align*}
Therefore, there exists $k_0\in \{0,\dots,m-1\}$ such that
\begin{align*}
\sum\limits_{\substack{(i,j)\in n\Omega_{m_1,m_2/m}^-
}}  \frac{ 1}{n} h_{v_{n}^{k}}^{ij} = \frac{1}{m}
\sum\limits_{\substack{(i,j)\in n\Omega_{m_1,m_2}^-
}}  \frac{ 1}{n} h_{v_{n}^{ij}} + O(m/n).
\end{align*}
Again the claim follows by taking the $\liminf$ as $n\rightarrow \infty$.
Hence, for all $m\in \mR\setminus \{0\}$, we deduce that
\begin{equation}
\label{eq:average}
\begin{split}
\frac{1}{m} C_-(V_1,V_2,m_1,m_2) &= C_-(V_1, V_2, \frac{m_1}{m}, \frac{m_2}{m})
\leq C_-(V_1,V_2, m_1, \frac{m_2}{m})\\
& \leq \frac{1}{m} C_-(V_1, V_2, m_1, m_2).
\end{split}
\end{equation}
Here the first inequality follows from monotonicity in $m_1$ and the second one from averaging. 
Thus, equality holds in all estimates in (\ref{eq:average}). In particular, for all $m\in \mR\setminus \{0\}$
\begin{align*}
C_-(V_1,V_2, \frac{m_1}{m}, m_2) = C_-(V_1, V_2, m_1, m_2),
\end{align*}
which yields the independence of $C_-(V_1,V_2,m_1,m_2)$ of $m_1$.
\end{proof}

As a consequence, the limiting energies only depend on the corresponding normal direction to the interface by means of the constants $C_\pm$, but \emph{do not} depend on the extension $m_1$ of the domain in the direction $(\pm1,\,1)$.
Similarly as Proposition \ref{prop:rig2}, this already partially confirms the expectation that the continuum energy will be a ``line energy''. 
Hence, after a normalization step, it is always possible to assume that the given layer energy is defined in a unit parallelogram.
Relying on Proposition \ref{prop:twell} from the previous section, we also obtain the following \emph{vertical cutting mechanism}:

\begin{prop}[Vertical cutting]
\label{prop:reduc}
Let $d,l> 0$ and let $u_n\in\mathcal{A}_{n \Omega_{2d,2l}^{-}}$ be piecewise affine on the grid $\Omega_n$. Suppose that 
\be
\label{eq:energy}
\sum\limits_{\substack{(i,j)\in n\Omega_{2d,2l}^-}}   \frac{1}{n^2} (\phi_{u_n}^{ij}  +  |\nabla_n \phi_{u_n}^{i,\,j}|)\leq \eta < \infty,
\ee
where $\phi\in C^{0,1}((\mR^2)^4, \mR)$ and 
\bes
\phi^{ij}_{u_n}:= \phi(\p_1 u_n^{ij}, \p_2 u_n^{ij}, \p_1 u_n^{i-1,j}, \p_{2} u_n^{i-1,j})
\ees
denotes a one-well energy function with well given by one of the sets $SO(2)U_i,\,i=1,2$, and quadratic growth, i.e.
\be
\label{eq:growthI}
c_1\dist^2(\nabla u^{ij}_n, SO(2)U_{i}) \leq \phi^{ij}_{u_n} \leq c_2\dist^2(\nabla u^{ij}_n, SO(2)U_{i}).
\ee
Then, there exist a constant $C>0$ and a modified deformation $\tilde{u}_n\in \mathcal{A}_{n\Omega_{2d,l/2}^{-}}$ such that along a subsequence
\begin{itemize}
\item[(a)]
$
H_{n}(\tilde{u}^{i,j}_n)  \leq C H_{n}(u^{i,j}_n)   \mbox{ for } (i,j)\in n\Omega_{2d, l/2}^-,
$
\item[(b)] for $x\in \Omega_{2d,l/2}^-$ it holds
\begin{align*}
\nabla \tilde{u}_n(x) &= U_0 \mbox{ for } x\cdot \begin{pmatrix} -1 \\ 1 \end{pmatrix}\geq \frac{3}{8}n d \mbox{ and } \\
\tilde{u}_n(x) &= u_n(x) \mbox{ for }  x\cdot \begin{pmatrix}- 1 \\ 1 \end{pmatrix}\leq \frac{1}{4}n d.
\end{align*}
\end{itemize}
\end{prop}

\begin{remark}
The previous ``cutting result'' will play a major role in our $\Gamma$-convergence proof (in the construction of the recovery sequence). We emphasize that for our proof it is necessary to pass from the larger domain $\Omega_{2d,2l}$ to the smaller set $\Omega_{2d,l/2}$ in the formulation of Proposition \ref{prop:reduc}. However, this does not pose difficulties in the proof of the $\Gamma$-convergence result, as we can exploit the scaling behavior of the boundary and layer energies which was formulated in Lemma \ref{lem:indep}.
\end{remark}

\begin{proof} 
The proof of the cutting lemma follows along the lines of Proposition 5.2 in \cite{CS06} and \cite{CS06c}. During the procedure in which we modify $u_n$ to $\tilde{u}_n$, we however have to ensure admissibility. This corresponds to two requirements: Firstly, we have to preserve impenetrability. Secondly we also have to make sure that the final function $\tilde{u}_n$ is still defined on the original lattice $(\Omega_n, \Delta^{n,\pm}_{i,j}$.\\

\emph{Step 1: Energy estimates.} There exist many values of $c_0\in [\frac{1}{4}n,\frac{3}{8}n]  \cap \Z$ such that for all $\delta \in (0,1)$
\begin{align*}
\frac{1}{\delta} \sum\limits_{\substack{(i,j)\in n\Omega_{2d,2l}^-,\\  c_0 - [\delta n] \leq i- j \leq c_0}}  \frac{1}{n^2}(\phi_{u_n}^{i,\,j}  + |\nabla_n \phi_{u_n}^{i,\,j}|) \leq c \eta.
\end{align*}
This follows by a covering argument as in \cite{CS06}. \\

\emph{Step 2: Construction of the reference grid $G_R^n$.} The construction of the grid $G_R^n$ is similar as in \cite{CS06} but with respect to the direction $(1,-1)$. It refines dyadically with vertical distances which we denote by $h_k$. However, instead of refining up to \emph{infinite} order, we limit ourselves to \emph{finite} scales such that $h_k$ is larger or equal to $\frac{1}{n}$. We denote the finest scale by $h_{k_0}^n$, and assume that $h_{k_0}^n \in [\frac{1}{n}, \frac{100}{n}]$.\\
More precisely, we define $l_1:= 2 \sqrt{2}[l]$ and $h_1:= [2d]$. Furthermore, we then set $l_k:= 2^{-k}l_1$ and $h_k:= 2^{-k}h_1$ as long as $h_k\geq h_{k_0}^n$ and stop the refining procedure after that. Then we divide the line segment in $\Omega_{2d,2l}^-$ with coordinates $i_k-j_k:=c_0 - [h_k n]$ into equi-sized intervals which are arranged symmetrically with respect to the line $j=0$. The boundaries of the intervals constitute the vertices of the grid $G^n_R$. The grid if formed by connecting the vertices along neighboring lines. We remark that the degeneracy of the triangles depends on the ratio $l/d$.\\

\emph{Step 3: Energy scaling.} In this step, we prove the optimal scaling of the energy in the respective triangles. This follows from the discrete analoga of the arguments in \cite{CS06}. We apply Step 1 with $\delta = h_k$:
\begin{align}
\label{eq:grad}
 \sum\limits_{\substack{(i,j)\in n\Omega_{2d,2l}^-,\\  c_0 - [h_k n] \leq i- j \leq c_0}}  \frac{1}{n^2}(\phi_{u_n}^{i,\,j}  + |\nabla_n \phi_{u_n}^{i,\,j}  |) \leq c \eta h_k.
\end{align}
Therefore there exists a parameter $c_k \in [c_0 - [h_k n], c_0]\cap \Z$ such that
\begin{align*}
\sum\limits_{\substack{(i,j)\in n\Omega_{2d,2l}^-,\\  i-j =c_k }} \frac{1}{n}(\phi_{u_n}^{i,j}+ |\nabla_n \phi_{u_n}^{i,\,j}  |) \leq c \eta.
\end{align*}
From the previous two estimates we infer that:
\begin{itemize}
\item[(a)] 
We have 
\bes
\sum\limits_{\substack{(i,j)\in n\Omega_{2d,2l}^-,\\  i- j = c_k}}  \frac{1}{n}|\nabla_n \phi_{u_n}^{i,\,j}|\le c\eta.
\ees
Spelling this out and considering the diagonal derivatives, we in particular obtain
\bes
\sum\limits_{\substack{(i,j)\in n\Omega_{2d,2l}^-,\\  i- j = c_k}} | \phi_{u_n}^{i+1,\,j} -\phi_{u_n}^{i,\,j}| \leq c \eta.
\ees
\item[(b)] There exists $j_k\in [-n,n]\cap \Z$ such that for the point $(i_k,j_k)\in n\Omega_{2d,2l}^-$ with $i_k - j_k =c_k$ it holds 
$$\phi_{u_n}^{i_k,j_k}\leq c \eta. $$
\end{itemize}
Thus, combining these two points by writing out a telescope sum, we observe that on $i-j=c_k$
\begin{align*}
|\phi_{u_n}^{i,\,j} - \phi_{u_n}^{i_k,\,j_k}| \leq \sum\limits_{\substack{(i,j)\in n\Omega_{2d,2l}^-,\\  i- j = c_k}} |\phi_{u_n}^{i+1,\,j}- \phi_{u_n}^{i,\,j}  | \leq c \eta .
\end{align*}
Due to the bound on $\phi_{u_n}^{i_k,\,j_k}$ from (b), this implies an $L^{\infty}$ estimate along the whole strip $i-j=c_k$, $|j| \leq n$:
\begin{align}
\label{eq:Linfty}
|\phi_{u_n}^{i,\,j}| \leq c\eta.
\end{align}
Defining $S_c:=\{(i,j)| c_0 - [h_k n]\leq i-j \leq c_0, \ [d] \leq i+j \leq [d]+[h_k n]\}$ for any $d\in\{-n,\dots, n-[h_k n]\}$, and invoking Poincar\'e's inequality in combination with (\ref{eq:grad}) and (\ref{eq:Linfty}) then yields
\be
\label{eq:Poinc}
\sum\limits_{(i,j)\in S_c} \frac{1}{n^2} \phi_{u_n}^{i,\,j}\leq 2 \left( h_k^2 \max_{\substack{(i,j)\in n\Omega_{2d,2l}^-,\\  i- j = c_k,\\ | j| \leq n}}|\phi_{u_n}^{i,\,j}| + h_k \sum\limits_{( i,j)\in S_c} \frac{1}{n^2} |\nabla_n \phi_{u_n}^{i,\,j}  |  \right) \leq c \eta h_k^2.
\ee

\emph{Step 4: Construction of $\tilde{u}_n$.} \\
\emph{Step 4a: Construction of the perturbed grid, $G_P^n$.} This follows as in \cite{CS06}. As in the construction of the reference grid, we however only refine up to $h_k \sim n^{-1}$, and recall that this finest scale is denoted by $h_{k_0}^n$. \\
We recall the precise construction from \cite{CS06}. Here the perturbed grid $G_P^n$ is obtained from the grid $G_R^n$ by perturbations along rigid directions. We seek to apply Proposition \ref{prop:twell} so that all the resulting new grid edges are rigid (c.f. Remark \ref{rmk:rigid} in Section \ref{sec:twell}). In order to remain within $\Omega_{2d,2l}^{-}$ within this procedure, we restrict our construction to the subgrid which is fully contained in $\Omega_{2d,l}^{-}$. We begin by defining $\alpha= \frac{1}{10}c(d,l)$, where $c(d,l)>0$ should be thought of as a small constant dealing with the degeneracy of the grid. Then we enumerate the grid vertices of $G_R^n$ and denote them by $v_m$. We apply Proposition \ref{prop:twell} in a ball $B_m:=B_{\alpha h_k}(v_m)$, where $v_m$ is a vertex on the layer $i_k-j_k = c_0 - [n h_k]$. Then for two neighboring balls $B_m$, $B_{m'}$ there are many rigid points $(w_m, w_{m'})$ according to Proposition \ref{prop:twell}. Following \cite{CS06}, we now describe our choice of the new grid vertices by iteratively defining \emph{possible choices at step $m$}:
\begin{itemize}
\item \emph{Possible choices at step $0$}: These are all $w_m\in B_m$.
\item \emph{Possible choices at step $m$}: These are all points $w_m \in B_m$ such that $w_m$ forms a rigid pair with many points of all neighboring balls. 
\item \emph{Possible choices at step $m+1$}: These are all possible choices from step $m$ without those points $w \in B_{m'}$ with $m'>m$ such that $v_{m'}$ is a neighbor of the $v_m$ but $(w,w_m)$ is not rigid. 
\end{itemize}
As in \cite{CS06} we claim that this algorithm works (i.e. the set of possible choices in step $m$ always forms a set of positive measure) and yields a new set of vertices which defines our new grid $G_P^n$. We give a proof for this in Appendix \ref{app:alg}, c.f. Lemma \ref{lem:alg}.\\
As in \cite{CS06} interpolation on the resulting triangles leads to a new piecewise affine grid function, $v$, on the new grid $G_P^n$. Furthermore, as in \cite{CS06a}, Proposition \ref{prop:twell} and the bound (\ref{eq:Poinc}) further yield estimates of the type   
\begin{equation}
\label{eq:close_a}
\begin{split}
|\nabla v(T_m) -  Q_m U_0| &\leq C \frac{1}{h_k} \| \dist(\nabla u, K) \|_{L^2(T_m)},\\
|\nabla v(T_m)- Q_m U_0| &\leq C \sqrt{\eta},
\end{split}
\end{equation}
for some rotation $Q_m$ associated to each triangle $T_m$ of the grid $ G^n_P$ (which is spanned by neighboring vertices). Here we used (\ref{eq:Poinc}) and (\ref{eq:growthI}) to obtain the second estimate in (\ref{eq:close_a}) from the first one. The notation $\nabla v(T_m)$ refers to the gradient of $v$ in the interior of the triangle $T_m$ (we recall that $v$ is a piecewise affine function on the perturbed grid $G_P^n$).
Furthermore, the facts that two neighboring triangles $T_m, T_m'$ of $G^n_P$ share a common edge and that on both triangles $\nabla v$ has a controlled distance to the wells in $K$ (c.f. (\ref{eq:close_a})), imply that 
\begin{equation}
\label{eq:close_a1}
\begin{split}
|\nabla v(T_m)-  \nabla v(T_m')| &\leq C \frac{1}{h_k} \| \dist(\nabla u, K) \|_{L^2(T_m)},\\
| Q_m' U_0 - Q_m U_0| &\leq C \frac{1}{h_k} \| \dist(\nabla u, K) \|_{L^2(T_m)}.
\end{split}
\end{equation} 
We now modify $v$ into a function $\tilde{v}_n$ on the original discrete grid $\Omega_n$. To this end, we define $\tilde{v}_n$ as the interpolation of $v$ with respect to the grid $\Omega_n$ (for most of the triangles, the interpolated gradient $\nabla \tilde{v}_n$ will equal the original gradient $\nabla v$ as the triangles in $G_P^n$ are in general much larger than those in $\Omega_n$, due to the choice $h_{k_0}^n \sim n^{-1}$ and as $v$ is affine on these). This yields a function which is defined on
\begin{align*}
\Omega_{2d,l/2}\cap \conv\{(i/n,j/n): i-j\leq c_0 -[n h_{k_0}^n]\}.
\end{align*}
In the interpolation process we obtain new error terms at the interfaces of two triangles in $G_P^n$, since the grid $G_P^n$ does not match the original grid $\Omega_n$. However, for these new interpolations, we note that $\nabla \tilde{v}(\Delta_{ij}^{n,\pm}) \in \conv\limits_{l\in\{1,\dots,m\}}(\nabla v(T_l))$, where the index $l$ denotes all the involved neighboring triangles in $G_P^n$ (in particular the maximal number of involved triangles, $m$, is independent of $n$). But due to (\ref{eq:close_a1}) this error is controlled, e.g. in the case $\nabla \tilde{v}_n(\Delta_{ij}^{n,+}) = \lambda \nabla v(T_m) + (1-\lambda ) \nabla v(T_{m}')$, we have
\begin{align*}
\int\limits_{\Delta^{n,+}_{ij}} |\nabla \tilde{v}_n(\D_{ij}^{n,+}) - \lambda Q_m U_0 - (1-\lambda ) Q_m' U_0|^2 dx 
&\leq  4(\int\limits_{T_m} |\nabla v(T_m) - \lambda Q_m  U_0|^2 dx \\
&\quad+  \int\limits_{T_m'} |\nabla v(T_m') -  Q_m' U_0|^2 dx)\\
&  \leq C  \| \dist(\nabla u, K) \|_{L^2(T_m)}^2 \\
& \quad + C  \| \dist(\nabla u, K) \|_{L^2(T_m')}^2,
\end{align*}
where we have used (\ref{eq:close_a}).
Summing over all triangles hence yields that
\begin{align*}
H_n(\tilde{v}_n) \leq C H_n(u_n).
\end{align*}
Moreover, due to the second estimate in (\ref{eq:close_a}), we note that $v$ and similarly $\tilde{v}_n$ satisfy the non-interpenetration condition.\\

\emph{Step 4b: Estimates on the original grid $\Omega_n$ close to the line $i-j=c_0$.} We estimate the contributions of $\nabla u$ on the original grid $\Omega_n$ in the domain given by $\Omega':=\{(i,j): |j|\leq n, c_0 - [10 h_{k_0}^n n] \leq i -j \leq c_0\}$. In contrast to the argument in the previous steps we do not construct a perturbed grid but seek to obtain estimates on the closeness of $\nabla u_n^{ij}$ to $SO(2)U_0$ on each individual grid triangle $\Delta_{ij}^{n,\pm}$. 
For this we use the one-well rigidity theorem of \cite{FJM05} together with (\ref{eq:grad}).
In particular, on the scale $h_k \sim n^{-1}$ these immediately yield pointwise bounds and we infer that
\begin{equation}
\label{eq:hsmall}
\begin{split}
|\phi^{i,j}_{u_n}| \leq c \eta \mbox{ on each triangle } \Delta_{ij}^{n,\pm}\subset \Omega'.
\end{split}
\end{equation}
Therefore, there exist rotations $Q_{ij}$ with
\begin{equation}
\label{eq:close_b}
\begin{split}
|\nabla u_n(\Delta_{ij}^{n,\pm})- Q_{ij} U_0| &\leq C \sqrt{\eta} \mbox{ for all triangles } \Delta_{ij}^{n,\pm}\subset \Omega',\\
|Q_{kl}U_0- Q_{ij} U_0| &\leq C \sqrt{\eta} \mbox{ for neighboring triangles } \Delta_{ij}^{n,\pm}, \Delta_{kl}^{n,\pm}\subset \Omega',\\
\sum\limits_{(i,j)\in n\Omega'} \frac{1}{n^2}h_{u_n}^{ij} &\leq C\int\limits_{\Omega'}\dist(\nabla u, K)^2 dx,
\end{split}
\end{equation}
where the last line follows from the one-well rigidity result, the observation that $\dist(\nabla u, SO(2)U_0)\leq \dist(\nabla u, K)$ on $\Delta_{ij}^{n,\pm}\subset \Omega'$ and the two-growth behavior of $h^{ij}_{u_n}$ close to the energy wells.\\

\emph{Step 4c: Construction of the interpolation function.} Using the estimate from Steps 4a and 4b, we now construct an interpolation function $w_n: \Omega_{2d,l/2}^{-} \rightarrow \mR^2$ between $u_n$ and $\tilde{v}_n$:
\begin{align*}
w_n(x):= \gamma(n x) u_n(x) + (1-\gamma)(n x) \tilde{v}_n(x)
\end{align*}
where $\gamma$ is a smooth function with $\gamma(z) = 1$ for $z_1-z_2 \in [c_0/n - [2 h_{k_0}^n n]/n, c_0/n]$,  $\gamma(z) = 0$ for $z_1-z_2\leq  c_0/n - [9 h_{k_0}^n n]/n$. Here, for completeness, $u_n$ and $\tilde{v}_n$ are set to equal zero in the domains in which they have not yet been defined. We claim that the resulting function $w_n$ satisfies the following energy bound:
\begin{align}
\label{eq:interp}
H_n(w_n) \leq CH_n(u_n).
\end{align}
Indeed, for $i-j\geq c_0 -[2h_{k_0}^n n]$ and for $i-j\leq c_0 - [9 h_{k_0}^n n]$ this follows from the respective bounds for $\tilde{v}_n$ and $u_n$ which were stated in Steps 4a and 4b. It thus remains to argue that this is also true in the interpolation region $c_0 - [9 h_{k_0}^n n] \leq i+j \leq c_0 - [2 h_{k_0}^n n] $. To this end, we note that as $h_{k_0}^n \sim n^{-1}$, $\nabla \tilde{v}(\Delta_{ij}^{n,\pm}) \in \conv(\nabla u(\D_{kl}^{n,\pm}))$, where $\D_{kl}^{n,\pm}$ are neighboring triangles of $\D_{ij}^{n,\pm}$ (or triangles within a certain uniformly bounded distance from $\D_{ij}^{n,\pm}$). As a consequence, by the triangle inequality and the estimates (\ref{eq:close_a}), (\ref{eq:close_a1}), (\ref{eq:close_b}), we infer
\begin{align}
\label{eq:close}
|\nabla u_n(\Delta_{ij}^{n,\pm}) - \nabla \tilde{v}_n(\D_{ij}^{\pm})| \leq C \frac{1}{h_k} \sum\limits_{(kl) \in \mathcal{N}(i,j)} \| \dist(\nabla u_n, K) \|_{L^2(\D_{kl}^{\pm})}.
\end{align}
Assuming growth of order two for the energy density $h_n$ at infinity, setting
\begin{align*}
\Omega'':=\{(x_1,x_2): x_1-x_2\in[c_0/n-[9h_{k_0}^n n]/n,c_0/n], \ |x_2| \leq 1\},
\end{align*}
and using (\ref{eq:close}), we hence obtain
\begin{align*}
H_n(w_n) &= H_n(u_n + \gamma (\tilde{v}_n-u_n)) \\
& \leq CH_n(u_n)  + \sum\limits_{(i,j)\in n\Omega''} \frac{1}{n^2}|\nabla u_n^{ij} - \nabla \tilde{v}_n^{ij}|^2  \\
& \quad + C \sum\limits_{\substack{(i,j)\in n\Omega_{2d,l/2}^-,\\  c_0 - [9 h_{k_0}^n n] < i- j < c_0 - [2 h_{k_0}^n n]}} \frac{1}{n^2}|\nabla \gamma|^2|u^{ij}_n-\tilde{v}^{ij}_n|^2  \\
& \leq C  H_n(u_n) + \sum\limits_{(i,j)\in n\Omega''} \frac{1}{n^2}|\nabla u_n^{ij} - \nabla \tilde{v}_n^{ij}|^2 + \sum\limits_{\substack{(i,j)\in n\Omega_{2d,l/2}^-,\\  c_0 - [h_k n] < i+ j < c_0}}\frac{1}{n^2}n^2|u^{ij}_n-\tilde{v}^{ij}_n|^2 \\
&  \leq  C H_n(u_n)+ C\sum\limits_{\substack{(i,j)\in n\Omega_{2d,l/2}^-,\\  c_0 - [h_k n] < i- j < c_0}}\frac{1}{n^2}|\nabla u^{ij}_n-\nabla \tilde{v}^{ij}_n|^2\\
&  \leq  C  H_n(u_n) ,
\end{align*}
where we used (\ref{eq:close}) in passing from the third to the fourth line and Poincar\'e's inequality to estimate the term involving $|u_n-\tilde{v}_n|$. 
Moreover, due to the $L^{\infty}$ estimates in (\ref{eq:close_a}), (\ref{eq:close_a1}) and (\ref{eq:close_b}) the admissibility of $w_n$ is preserved.
Hence, setting $\tilde{u}_n=w_n$, provides the desired modification of $u_n$.
\end{proof}

An analogous cutting result holds for the limiting profiles $v^+_{V_1,V_2}$.\\

\subsection{Proof of the $\Gamma$-convergence result}
\label{sec:Gamma}

In this section we finally prove the $\Gamma$-convergence result of Theorem \ref{Th1}. Here the $\Gamma$-$\liminf$ inequality essentially follows directly from the definition of the limiting energy and the independence result of Lemma \ref{lem:indep}. The construction of the $\Gamma$-$\limsup$ inequality however is more involved (as in \cite{CS06a}, \cite{CS06}, \cite{CS06c}). Here we have to invoke the cutting result of Proposition \ref{prop:reduc}.

\begin{proof}[Proof of the $\Gamma-\liminf$ inequality]
Using the definition (\ref{eq:layer_energy1}), the $\Gamma-\liminf$ inequality follows directly: Without loss of generality we may assume that $\liminf\limits_{n\rightarrow \infty}H_n^1(u_n)\leq  C < \infty$.  In this setting, the compactness result of Proposition \ref{prop:rig2} holds. Thus, along a subsequence, we obtain a limiting deformation $u_0$ which is a simple laminate. In particular, its gradient attains values in $K= SO(2)U_0 \cup SO(2)U_1$ and only has finitely many, say $L\in \N$, jump interfaces. Furthermore, we claim that it suffices to assume that the jump interfaces of $\nabla u_0$ \emph{do not} intersect on $\partial \Omega$. Indeed, this follows from the observation that if there were intersections of jump interfaces on the boundary, then we could carry out the procedure which is described below in domains which slightly stay away from the boundary. More precisely, for any given $\epsilon>0$ we would only cover an $(1-\epsilon)$ fraction of the jump set (i.e. only the interior parts of the jump set, which are at distance $\epsilon/2$ away from the boundary) by the sets $\Omega_n^k$ which are described below. By virtue of the arbitrariness of $\epsilon$ this yields our claim.\\

With this discussion in mind (in particular assuming that the interfaces are separated from each other and do not intersect on $\partial \Omega$), we now cover the jump set of $\nabla u_0$ by subdomains $\Omega_k^n$, of which each only contains a single jump interface or one of the boundary layers given by the points $(x_1,x_2)$ with $x_1+x_2 \leq -4$ and $x_1+x_2\geq 4$ and $|x_2 |\leq 2$. We consider associated subenergies $H_n^{1,k}(\cdot)$ determined by the sets $\Omega_k^n$ and the interfaces and boundaries of the limiting configuration $u_0$:
\begin{align*}
H^1_n(u_n) =  H_n^{1,0}(u_n \chi_{\Omega_k^n}) + \sum\limits_{k=1}^{L-1} H_n^{1,k}(u_n \chi_{\Omega_k^n}) +  H_n^{1,L}(u_n \chi_{\Omega_L^n}).
\end{align*}
By the compactness result of Proposition \ref{prop:rig2}, there exist points $x_k^n\in\Omega$ with $x_k^n \rightarrow x_k\in \Omega$ such that up to subsequences
\begin{align}
\label{eq:conv_1}
u_n(\cdot - x_k^n) \chi_{\Omega_k^n}  \rightarrow  v_{V_k,V_{k+1}}^{\pm} \mbox{ in } L^{1}(\Omega_k^n),
\end{align}
where the functions $v_{V_k,V_{k+1}}^{\pm}$ are defined in Definition \ref{defi:profile}. We further observe that for all $k\in\{0,\dots,L\}$ and for each $\epsilon>0$ there exists $N_{\epsilon,k}\in \N$ such that for all $n\geq N_{\epsilon,k}$
\begin{align*}
H^1_k(u_n \chi_{\Omega_k^n}) \geq \liminf\limits_{n\rightarrow \infty} H^1_k(u_n \chi_{\Omega_k^n}) - \epsilon.
\end{align*}
Then however, with $\epsilon>0$ arbitrary but fixed and $N_{\epsilon}:= \max\limits_{k\in\{0,\dots,L\}}N_{\epsilon,k}$, we immediately infer that for $n\geq N_{\epsilon}$ (where we invoke the independence result of Lemma  \ref{lem:indep} and (\ref{eq:conv_1})):
\begin{align*}
H^1_n(u_n) &= H_n^{1,0}(u_n \chi_{\Omega_k^n}) + \sum\limits_{k=1}^{L-1} H_n^{1,k}(u_n \chi_{\Omega_k^n}) +  H_n^{1,L}(u_n \chi_{\Omega_L^n})\\
& \geq   H_n^{1,0}(u_n \chi_{\Omega_k^n}) - \epsilon + \sum\limits_{k=1}^{L-1} \left(\liminf\limits_{n  \rightarrow \infty} H_n^{1,k}(u_n \chi_{\Omega_k^n}) - \epsilon \right) +  H_n^{1,L}(u_n \chi_{\Omega_L^n)}-\epsilon \\
& \geq  \sum\limits_{k=1}^{L-1} \inf\{ \liminf\limits_{n  \rightarrow \infty} H_n^{1,k}(v_n \chi_{\Omega_k^n}) , v_n  \rightarrow v_{V_k,V_{k+1}}^{\pm} \mbox{ in } L^1(\Omega_k^n)\} \\
& \quad + \inf\{ \liminf\limits_{n  \rightarrow \infty} H_n^{1,0}(v_n \chi_{\Omega_0^n}) , v_n  \rightarrow v_{F_{\lambda},V_{1}} \mbox{ in } L^1(\Omega_0^n), \\
& \quad \quad \quad v_n(i/n,j/n)=F_{\lambda}(i/n,j/n) \mbox{ for } i+j\leq -4n\} \\
& \quad + \inf\{ \liminf\limits_{n  \rightarrow \infty} H_n^{1,L}(v_n \chi_{\Omega_0^n}) , v_n  \rightarrow v_{V_L,F_{\lambda}} \mbox{ in } L^1(\Omega_0^n),\\
& \quad \quad \quad v_n(i/n,j/n)=F_{\lambda}(i/n,j/n) \mbox{ for } i+j\geq 4n\} - (L+1)\epsilon \\
&\geq E_{surf}(u_0) - \epsilon.
\end{align*}
In the second last inequality, we carried out a translation of $u_n$ in order to match the boundary conditions for the right boundary layer.
Since this estimate holds for any $\epsilon>0$, this concludes the proof of the $\Gamma-\liminf$ inequality.
\end{proof}

We now proceed to the proof of the $\Gamma$-limsup inequality. As in Conti and Schweizer \cite{CS06}, \cite{CS06a}, \cite{CS06c} this is the harder part of the argument. In the presence of multiple interfaces we have to cut and paste the different internal and boundary layers, which are provided by the minimization problem that defines the densities of $E_{surf}(\cdot)$. This has to be achieved in a way which leads to an overall admissible sequence. In particular, we have to preserve the non-interpenetration condition. To ensure these issues, we rely on the cutting procedure from Proposition \ref{prop:reduc}.

\begin{proof}[Proof of the $\Gamma-\limsup$ inequality]
For the purpose of this proof, we introduce the following abbreviation:
\begin{align}
\label{eq:en_abb}
H^{1,\pm}_{n}(v,m_1,m_2):= \sum\limits_{(i,j)\in n\Omega_{ m_1,m_2}^{\pm}}\frac{1}{n}h_{v}^{ij},
\end{align}
to denote the energies which were used as a building block in Definition \ref{defi:boundarylayers} and Lemma \ref{lem:indep}. \\

\emph{Step 1: Reduction to Proposition  \ref{prop:reduc}.}
Given $u_0$ which is piecewise affine with gradient in $K=SO(2)U_0  \cup SO(2)U_1 $ in $\Omega$, we have to construct a sequence of $u_n$ which is admissible, converges to $u_0$ in $L^1(\Omega)$ and satisfies
\begin{align*}
\limsup\limits_{n \rightarrow \infty} H_n^1(u_n) \leq  E_{surf}(u_0).
\end{align*}
As $E_{surf}(u_0)$ is defined by a sum of boundary and internal layer energies and as each of these is determined by a minimization process (c.f. (\ref{eq:layer_energy1}), (\ref{eq:total_energy1})), for each jump interface of $\nabla u_0$ we find subsequences $n_j$ and $u_{n_j}^{k}:\Omega_{d,4}^{-} \rightarrow \mR^2$ such that (with the notation from (\ref{eq:en_abb})) for instance
\begin{align*}
\lim\limits_{j  \rightarrow \infty} H_{n_j}^{1,-}(u_{n_j}^k,d,4) = C_-(V_k,V_{k+1},d,4) = 4C_-.
\end{align*}
In the sequel, we concentrate on this single jump interface; the results for the other internal and boundary layers follow analogously.
We seek to modify these functions $u_{n_j}^k$ into new functions $\tilde{u}_{n_j}^k$ such that they are defined in (part of) our original domain $\Omega$ and have affine boundary data. Then, if we can extend the functions $\tilde{u}_{n_j}^k$ to a full sequence in $n\in \N$ (not just the subsequence $\{n_j\}_{j\in\N}\subset \N$; this is done in Step 3), then the affine boundary data would allow us to glue the individual pieces together. This would hence yield a global recovery sequence defined on $\Omega$.\\
Returning to our interface with orientation $(-1,1)$ between the gradients $V_k, V_{k+1}$, we claim that there is a sequence $\tilde{u}_n^k$ (derived from the function $u_ {n_j}^k$) such that in $\Omega_{d,1}^-$ we have
\begin{itemize}
\item[(a)] $\tilde{u}_n^k  \rightarrow u_0 \mbox{ in } L^1(\Omega_{d,1}^-)$,
\item[(b)] $\tilde{u}_n^k$ is affine away from the interface, more precisely there are orientation preserving isometries $I_n$, $I_{n}'$
\begin{align*}
u_n^k(x,y) = \left\{ \begin{array}{ll}
I_n \circ u_0 \mbox{ for } x-y \geq 4/5 d,\\
I_n' \circ u_0 \mbox{ for } x-y \leq -4/5 d,
\end{array}  \right.
\end{align*}
and $I_n, I_n'  \rightarrow Q, Q'\in SO(2)$,
\item[(c)] $H^1_n(\tilde{u}_n^k) = H_{n}^{1,-}(\tilde{u}_{n}^k,d,1)  \rightarrow C_-(V_k, V_{k+1},1,1)=C_-.$
\end{itemize}
The previous claims (a)-(c) are deduced by an application of Proposition \ref{prop:reduc}. 
In order to do so we first observe the independence of $C_-(V_k, V_{k+1},m_1,m_2)$ on the extension of the domain $\Omega_{m_1,m_2}^-$
in the direction orthogonal to the interface (c.f. Lemma \ref{lem:indep}). 
Next, by the definition of $E_{surf}(u_0)$ and by Lemma \ref{lem:indep}, we directly infer that for each $\eta>0$ there exists a number $N_{\eta}\in \N$ such that
\begin{align}
\label{eq:energy_control}
\sum\limits_{\substack{(i,j)\in n\Omega_{d,1}^-,\\  [n/2] \leq |i-j| \leq [n]}}  \frac{1}{n} h_{u_n^k}^{ij} \leq C \eta \mbox{ for all } n\geq N_{\eta}.
\end{align}
Let us then choose $d,l$ in the assumptions of Proposition \ref{prop:reduc} so that
\beas
&&n\Omega_{2d,l/2}^-=\{(i,j)\in n\Omega_{d,1}^-,\ \  [n/2] \leq i-j \leq [n]\}\\
&&\mbox{ or } n\Omega_{2d,l/2}^-=\{(i,j)\in n\Omega_{d,1}^-,\ \  -[n] \leq i-j \leq -[n/2]\}.
\eeas
Finally,  in Step 2 below we construct the one-well energy satisfying \rf{eq:energy}. Combining these observations allows us to apply Proposition \ref{prop:reduc} and hence to replace our minimal sequence $u_{n_j}^k$ by the corresponding
modification $\tilde{u}_{n_j}^{k}$.
The resulting sequence $\tilde{u}_{n_j}^k$ satisfies an analogous energy bound and consequently yields statements (a)-(c) from above. In particular, its boundary data are affine and lie in the respective energy wells. 
These affine boundary data then permit us (after a suitable translation) to glue together the individual functions $\tilde{u}_{n_j}^k$, $k\in\{1,\dots,L\}$, which were obtained for the individual interfaces. Hence, it remains to construct the one-well energy satisfying \rf{eq:energy}. This is the content of the next step.\\ 

\emph{Step 2: Reduction to a one-well energy.}
Seeking to apply the two-well rigidity result of Proposition \ref{prop:twell}, we construct a one-well energy density which satisfies the necessary bounds.\\

We begin by considering the following one-well energy density 
\begin{align}
\label{1WE}
\phi^{ij}_u := \gamma (|U^{i,j}|) \min\{\bar{h}_{u,U_0}^{i,j},\bar{c}/10\} + (1-\gamma)(|U^{i,j}|) k(|U^{ij}|).
\end{align}
Here $\bar{h}_{u,U_0}^{ij}$ denotes the one-well function from Remark \ref{rmk:bracket}, $\bar{c}$ is the constant from (\ref{cb}), $U^{i,j}:=(\partial_1 u^{i,j},\partial_1 u^{i-1,j},\partial_2 u^{i,j},\partial_2 u^{i,\,j-1})$. The function $\gamma\in C^\infty( \mR,\mR)$ is a cut-off function with $\gamma(t)=1$ for all $|t| \leq 10\max\{10\bar{c},100\}$ and $\gamma (t)= 0$ for $|t|\geq 20\max\{10\bar{c},100\}$.
Moreover, $k$ is chosen such that 
\begin{align}
\label{eq:2growthI}
c_1 \dist(\nabla u^{ij}, SO(2)U_0)^2 \leq k(U^{ij}) \leq c_2 \dist(\nabla u^{ij},SO(2)U_0)^2,
\end{align}
for some constants $c_1, c_2 >0$. Similarly as in Remark \ref{rmk:bracket} we interpret $\phi_u$ as the composition of a Lipschitz continuous function $\phi_{\cdot}$ with the piecewise constant function $\nabla u$.\\
We claim that for any $u\in \mathcal{A}_n$
\begin{align}
\label{eq:bound_one_well}
|\nabla_n \phi^{ij}_u| \leq C n h^{ij}_u,
\end{align}
where $h^{ij}_u$ denotes our original model Hamiltonian from Definition \ref{defi:HamiltonianII}.
Indeed, as 
$$|\nabla_n \phi^{ij}_u| \leq Cn(|\phi^{i+1,j}_u - \phi^{i,j}_u| + |\phi^{i,j+1}_u- \phi^{i,j}_u|),$$ 
(\ref{eq:bound_one_well}) directly follows in the region where $\dist(\nabla u^{ij}, SO(2)U_0)^2 \leq \bar{c}/100$ from (\ref{1WE}). This is due to the fact that in this region, in (\ref{1WE}) the first summand and there the first term in the bracket is active. But then for $(i,j)\in n \Omega$ with $\dist(\nabla u^{ij}, SO(2)U_0)^2 \leq \bar{c}/100$, this first term in the bracket is controlled by our original Hamiltonian $h^{ij}_u$ (as also here the first bracket is active while the second one is bounded below). \\
In the region where $\dist(\nabla u^{ij}, SO(2)U_1) \leq \dist(\nabla u^{ij}, SO(2)U_0) $ and $|\nabla u^{ij}|\leq \max\{100, 10\bar{c}\}$, the function $\phi^{ij}_u$ is constant, hence $0=|\nabla_n \phi_u^{ij}| \leq h_u^{ij} $. Thus, by compactness (and by the Lipschitz regularity of $\phi_{\cdot}$), the bound (\ref{eq:bound_one_well}) follows for all values of $\nabla u^{ij}$ with $|\nabla u^{ij}|\leq  \max\{100, 10\bar{c}\}$. Finally, for $|\nabla u^{ij}|\geq  \max\{100, 10\bar{c}\}$, the bound follows from the two-growth assumption (\ref{eq:2growthI}) which is satisfied by both $h^{ij}_u$ and $\phi^{ij}_u$. \\

Using (\ref{eq:bound_one_well}), we thus infer
\begin{align*}
&\sum\limits_{\substack{(i,j)\in n\Omega,\\  -[n/2] \leq | i- j| \leq [n]}}\frac{1}{n^2}|\nabla_n \phi^{ij}_u|
\leq C \sum\limits_{\substack{(i,j)\in n\Omega,\\  -[n/2] \leq | i- j| \leq [n]}}\frac{1}{n}  h^{ij}_u \leq c \eta.
\end{align*}
This then permits us to invoke Proposition \ref{prop:twell}.\\

Hence, on the level of our subsequence $u_{n_j}^k$ we have obtained a recovery sequence. In order to pass to a full sequence, we invoke a scaling argument as in \cite{CS06}.\\

\emph{Step 3: Passage from the subsequence $n_j$ to a full sequence in $n\in \N$.}
The proof of the extension of the recovery sequence from a subsequence to a full sequence follows along the argument given by Conti and Schweizer \cite{CS06a}. It relies on a combination of a scaling argument and the energy control from (\ref{eq:energy_control}). As before we restrict our attention to a single interface which has a normal pointing into the $(-1,1)$ direction. For the situation with more interfaces we argue locally around each interface.\\

\emph{Step 3a: Scaling.}
We claim that for each $n\in \N$ there exists a function $v_n:\Omega_{d,1}^-\rightarrow \mR^2$ such that (with the abbreviation from (\ref{eq:en_abb}))
\begin{align*}
\limsup\limits_{n\rightarrow \infty} H^{1,-}_n(v_n,\infty,1)\leq C_-, 
\end{align*}
and there exists $L_n>0$ such that 
\begin{align*}
\nabla v_n(x_1,x_2)& = U_0 \mbox{ for } x_1 - x_2 \geq L_n,\\
\nabla v_n(x_1,x_2)&= QU_1 \mbox{ for } x_1 - x_2 \leq -L_n.
\end{align*}

The claim follows from scaling. Indeed, by the definition of $C_-(U_0, QU_1, d,1)$ there exist sequences $n_j$ and $u_{n_j}$ such that
\begin{align*}
H^{1,-}_n(u_{n_j},d,4) \rightarrow 4 C_- \mbox{ and } u_{n_j} \rightarrow v_{U_0, QU_1}^{-} \mbox{ in } L^1(\Omega_ {d,4}^{-}).
\end{align*}
By Proposition \ref{prop:reduc} this implies that there exists a sequence $\tilde{u}_{n_j}$ with affine boundary data such that
\begin{align*}
H^{1,-}_n(\tilde{u}_{n_j},d,1) \rightarrow C_- \mbox{ and } \tilde{u}_{n_j} \rightarrow v_{U_0, QU_1}^{-} \mbox{ in } L^1(\Omega_ {d,1}^{-}).
\end{align*}
Let $\epsilon_j:= |H^{1,-}_n(\tilde{u}_{n_j},d,1) - C_- |$ denote the error at stage $j$ and assume that $n_j$ is a monotone increasing sequence. Then, for given $n\in \N$, let $n_j\in \N$ be the smallest element in $\{n_j\}_{j\in\N}$ such that $n^2 < n_j$ and abbreviate $\alpha:=\frac{n_j}{n}>n$. Moreover, define
\begin{align*}
\tilde{v}_n(i,j):= \alpha \tilde{u}_n(i/\alpha, j/\alpha).
\end{align*}
Thus, the scaling of the energy (c.f. proof of Lemma \ref{lem:indep}) yields
\begin{align*}
H_n^{1,-}(\tilde{v}_n,\alpha d,\alpha) = \alpha H^{1,-}_n(\tilde{u}_{n_j},  d,1) \leq \alpha C_{-} + \alpha \epsilon_j + \frac{c}{\alpha n}. 
\end{align*}
By translating, it is possible to find a point $(\bar{i},\bar{j})\in \Omega_{\alpha d,  \alpha}^-$ with distance $d$ away from the boundary such that for $v_n(i,j):=\tilde{v}_n(i,j+\bar{j})$ we have
\begin{align}
\label{eq:en_con}
H_n^{1,-}(v_n, \alpha d, 1) =  H^{1,-}_n(\tilde{u}_{n_j}, d, 1) \leq  C_{-} +  \epsilon_j + \frac{c}{\alpha^2 n}. 
\end{align}
As by construction $\nabla v_n^{ij}$ is in the energy wells if $i-j \geq \alpha$ or $i-j\leq -\alpha$, this proves the claim with $L_n:=\alpha$.\\

\emph{Step 3b: Energy bounds.}
We claim that there exist $h>0, L>2h, \delta>0$, all independent of $n$, a function $w_n:\Omega_{L,d}^{-} \rightarrow \mR^2$ such that
\begin{align*}
\limsup\limits_{n \rightarrow \infty} H^{1,-}_n(w_n,L,1) \leq C_-,
\end{align*}
and
\begin{itemize}
\item[(i)] for half of all points $(i,j)$ with $i-j\in(Ln-hn, Ln)$ 
\begin{align*}
\essinf\limits_{j\in(-n,n)\cap \Z^2}\dist(\nabla w_n^{i,j},SO(2)U_1) \geq \delta,
\end{align*}
\item[(ii)] there exists a value $j_n \in (-Ln, Ln-2hn)$ (depending on $n$) such that for half of the points $(i,j)$ with $i-j\in (j_n,j_n+hn)$ it holds
\begin{align*}
\essinf\limits_{j\in(-n,n)\cap \Z^2}\dist(\nabla w_n^{i,j},SO(2)U_0) \geq \delta.
\end{align*}
\end{itemize}
This shows (in a weak form) that the transition from $SO(2)U_1$ to $SO(2)U_0$ already takes place in the smaller domain $\Omega_{L,d}^{-}$. The proof follows from the energy control in (\ref{eq:en_con}). More precisely, we choose $\delta<\dist(SO(2)U_0, SO(2)U_1)/10$ and consider
\begin{align*}
f_{U_0}(c)&:= \#(\{ (i,j) \in n \Omega_{\infty,1}^-: i-j=c, \ \dist(\nabla v_n^{i,j},SO(2)U_0)\leq \delta)\},\\
f_{U_1}(c)&:= \#(\{ (i,j) \in n \Omega_{\infty,1}^-: i-j=c, \ \dist(\nabla v_n^{i,j},SO(2)U_1)\leq \delta)\}.
\end{align*}
In order to prove the statement, we show that $f_{U_0}, f_{U_1}$ are essentially characteristic functions. For this we observe the following points:
\begin{itemize}
\item As $H_n^{1,-}(v_n, \infty, 1) \leq c$, we have 
\begin{align*}
&\#\{c:  f_{U_0}(c) + f_{U_1}(c)< \frac{3}{2}\}\leq c_1 n.
\end{align*}
\item As a transition from $SO(2)U_0$ to $SO(2)U_1$ costs a finite amount of energy (c.f. the argument in Lemma \ref{lem:lowerbound}), 
\begin{align*}
&\#\{c: f_{U_0}(c) \neq 0,  f_{U_1}(c)\neq 0\}\leq c_2n.
\end{align*}
\item It holds
\begin{align*}
f_{U_0}(c) = 2n \mbox{ for } c \geq n L_n, \ f_{U_0}(c) = 0 \mbox{ for } c \leq - n L_n.
\end{align*}
\item If $f_{U_0}(c_a)\geq \frac{3}{2}n$ and $f_{U_1}(c_b)\geq \frac{3}{2}n$, then for $\bar{v}_n(i,j):= v(i+i_c,j+j_c)$ where $i_c - j_c = \frac{c_a+ c_b}{2}$,
\begin{align*}
H_n^{1,-}(\bar{v}_n,|c_a-c_b|,1) \geq c .
\end{align*}
\end{itemize}
Combining this, we obtain that there exists sets $M_0,M_1,M_2$ of points $(i,j)$ such that $\# M_2 \leq (c_1+c_2)n^2$,
\begin{align*}
f_{U_0}(c) \geq \frac{3}{2}n, \ f_{U_1}(c)=0 \mbox{ for } (i,j)\in M_0 \mbox{ with } i+j =c,\\
f_{U_0}(c) = 0 ,\ f_{U_1}(c)\geq \frac{3}{2}n,  \mbox{ for } (i,j)\in M_1 \mbox{ with } i+j =c. 
\end{align*}
In other words, $M_0$ denotes the set of lines such that $SO(2)U_0$ is the preferred value of $\nabla u$ on these lines. $M_2$ plays the same role for lines on which $\nabla u$ is mostly in $SO(2)U_1$. Finally, $M_3$ denotes the ``mixed'' lines where both $SO(2)U_0$ and $SO(2)U_1$ appear in a large volume fraction.\\
We observe that the number of interfaces between $M_0,M_1$ is bounded by a constant $c_3$ (which is uniform in $n$). Defining $c_a:= \inf\{c\in \Z: \mbox{ for all } (i,j) \mbox{ with } i-j \in (c-j,c), \ (i,j)\notin M_1\}$ and choosing $h \geq 2(c_1+c_2)$ yields (ii) in the interval given by $(i,j)$ with $i-j\in (c_a-h,c_a)$. Choosing a large number $L$ with $L>(c_3+2)h$ and dividing the interval of points in which $i-j \in (c_a-L,c_a -h)$ into sections of size $h$, we note that by definition, all of them intersect $M_1$. For sufficiently large $L$ there exists one section which does not intersect $M_0$, which follows as the number of interfaces between $M_1$ and $M_0$ is bounded by $c_3$. This yields the existence of the desired value $c_n$ from (ii). The function $w_n$ is obtained by an appropriate translation of the function $v_n$.\\

\emph{Step 3c: Compactness and conclusion.}
Finally we construct the desired full sequence $u_n$ in $\Omega_{d,1}^-$. This sequence both satisfies the energy bound (\ref{eq:energy_control}) and converges against the desired limiting profile. To this end, we claim that there exists a rotation $R$, a point $a\in (-L+h/2, L-h/2)$ and a translation vector $b\in \mR^2$ such that
\begin{align}
\label{eq:conv_a}
\| R w_n(i + i_a, j+ j_a) + b - v_{U_0,Q U_1}^-(i/n,j/n) \|_{L^1(\Omega_{d,1}^{-})} \rightarrow 0.
\end{align}
Indeed, this follows from compactness. Assume that it were not the case. Then, by the boundedness of the energy of $w_n$ we can invoke Proposition \ref{prop:rig2} along the ``bad sequence'' which satisfies the energy bound (\ref{eq:energy_control}) but does not obey (\ref{eq:conv_a}) and obtain a limiting deformation $w_{\infty}$ with $\nabla w_{\infty}\in K$. For $i-j\leq -L$ and $i-j\geq L$ it attains the gradient values $R_1 U_0$ and $R_2 U_1$ with $R_1,R_2\in SO(2)$, respectively. Moreover, by Step 2b, the interface must have the normal $(-1,1)$. Hence, the limiting deformation $w_{\infty}$ involves \emph{at least one} interface between the energy wells and the corresponding interface has the right orientation. Furthermore, it cannot involve more interfaces, as these would cost a non-vanishing additional amount of energy (c.f. the proof of the $\Gamma$-$\liminf$ inequality). This yields a contradiction to our assumption that $w_n$ does not converge to the desired limiting profile after translation and rotation. Rotating and translating $w_n$ appropriately, yields the definition of $u_n$ and concludes the proof.
\end{proof}

\section*{Acknowledgments}
G.K. acknowledges a postdoctoral scholarship at the Max Planck Institute for Mathematics in the Science
during which this work was initiated as well as partial funding from the Leverhulme Trust grant,
"Liquid Crystal Defects in Landau-de Gennes theory", RPG-2014-226 leading to these results.
A.R. acknowledges that the research leading to these results has received funding from the European Research Council under the European Union's Seventh Framework Programme (FP7/2007-2013) / ERC grant agreement no 291053 and a Junior Research Fellowship at Christ Church.

\begin{appendix}

\section{Mapping the Microscopic Two-Well Problem to a Spin System}
\label{sec:spin}
In this section we map the two-well problem to a spin system and prove one-sided energy bounds which are crucially used in the compactness result of Section \ref{sec:compact}. In the whole section for convenience of notation we assume that the constant $\bar{c}$ from (\ref{cb}) is such that $\bar{c}\leq \bar{c}^2$. In this section, the discrete nature of our problem is strongly used. In this context, a number of arguments simplify with respect to the analogous continuous models. It is in this part of our argument that the origins of the surface energies in the discrete model appear for the first time.
We begin by introducing the corresponding definitions and abbreviations. 

\begin{defi}[Spin Hamiltonian]
\label{defi:conventions}
Let $\bar{c}$ be as in (\ref{cb}) and $u\in\mathcal{A}_n^{F_\lambda}$. Then we set 
$$\tilde{\Omega}_0^n:= \left\{(i/n,j/n) \in \Omega: \ h^{i,\,j}_{u_n} \leq \frac{\bar{c}}{10} \mbox{ and } \dist(\nabla u_n^{i,j} , SO(2)U_0) \leq \frac{\bar{c}}{10}   \right\},$$
and $\tilde{\Omega}_1^n:=\Omega_n \setminus \tilde{\Omega}^n_0$. 
We further define the \emph{discrete spin function} $\sigma_n: (n^{-1}\Z)^2 \cap \Omega\go \{\pm 1\}$ as
\begin{equation}
\label{eq:spin}
\begin{split}
\sigma_n^{i,j}:= \left\{ \begin{array}{ll} 
1 &\mbox{ if } (i/n,j/n)\in \tilde{\Omega}_0^n,\\
-1 &\mbox{ else,}
\end{array}
 \right.
\end{split}
\end{equation}
and correspondingly the \emph{spin Hamiltonian} as
\begin{align*}
  H_{n}^s(\sigma_n^{i,j}):= \sum\limits_{(i/n,j/n)\in G^n \cap \Omega}\sum\limits_{(k,l)\in  \{(i+1,j), (i, j+1), (i-1,j), (i,j-1)\}} n^{-2} (\sigma_n^{i,j}-\sigma_{n}^{k,l})^2.
\end{align*}
\end{defi}

Next, we claim the following one-sided comparability of the two-well and spin Hamiltonians.

\begin{prop}
\label{prop:comparison}
Let  $u\in \mathcal{A}_n^{F_\lambda}$.
Then there exists a constant $0<C=C(a,b)<\infty$ such that
\begin{equation}
\label{eq:Hamiltonian}
H_n(u_n)\geq C H_n^s(\sigma_n^{i,j}).
\end{equation}
\end{prop}

In order to prove this proposition, we first show the following auxiliary result:

\begin{lem}
\label{lem:lowerbound}
Let $u\in \mathcal{A}_n^{F_\lambda}$ and $(i/n,j/n)\in \tilde{\Omega}_0^n$ but assume that $((i+1)/n,j/n)\notin \tilde{\Omega}_0^n$. Then, 
\begin{align*}
h^{i+1,\,j}_u>\frac{\bar{c}}{100}.
\end{align*}
\end{lem}

\begin{proof}
Indeed, this follows from the two body interactions which are involved in the definition of $h^{ij}_{u}$ close to the two wells $SO(2)U_0\cup SO(2)U_1$. In order to see this, we argue by contradiction and assume that the conclusion of the lemma were false. This would entail that (in the notation of Remark \ref{rmk:bracket})
\begin{align*}
& \bar{h}^{ij}_{u, U_0} \leq \frac{\bar{c}}{4} \mbox{ and } \bar{h}^{i+1,j}_{u, U_1} \leq \frac{\bar{c}}{4}.
\end{align*}
However, this is not possible as the first assumption implies
\begin{align*}
(|\p_1 u^{i,j}|^2-a^2)^2 \leq \frac{\bar{c}}{4},
\end{align*}
while the second condition enforces
\begin{align*}
(|\p_1 u^{i,j}|^2-b^2)^2 \leq \frac{\bar{c}}{4}.
\end{align*}
As by virtue of the choice of $\bar{c}$ this is not possible simultaneously, we obtain a contradiction. Thus, we conclude the desired result.
\end{proof}

Using the previous lemma, we can proceed with the proof of Proposition  \ref{prop:comparison}:

\begin{proof}[Proof of Proposition  \ref{prop:comparison}]
By definition we have that $(\sigma_n^{i,j}-\sigma_{n}^{k,l})^2 \in \{0,4\}$ for each $(i/n,j/n)\in \Omega_n$ and any of its neighbors $(k/n,l/n) \in \Omega_n$. 
Thus, we only have to argue, that the number of points on which $h^{i,\,j}_u$ is uniformly bounded from below, e.g. by $\frac{\bar{c}}{100}$, is larger or equal to the number of points on which $(\sigma_n^{i,j}-\sigma_{n}^{k,l})^2 $ attains the value $4$. But this is ensured by Lemma  \ref{lem:lowerbound}.
\end{proof}

As a direct corollary of Proposition \ref{prop:comparison} and the energy bound (\ref{eq:surfscaling}) we obtain that for each $n\in\N$ the set of edges in $G^n$ which connects two vertices such that $(\sigma_n^{i,j}-\sigma_{n}^{k,l})^2 $ attains the value $4$ has a uniformly (in $n$) bounded one-dimensional Hausdorff measure. It divides $\Omega$ into two connected components $\Omega_0^n$ and $\Omega_1^n$ such that $\sigma_n^{i,j}=\pm 1$ for $(i/n,\,j/n)\in \Omega_{0,1}^n$, respectively. Both are Caccioppoli sets, whose perimeter is uniformly bounded in $n$. Moreover, we can interpolate the lattice function $\sigma^{ij}_n$ constantly and define a function $\sigma\in BV(\Omega)$ which is equal to $\pm 1$ for $x\in\Omega_{0,1}^n$, respectively.
Hence, by the compactness results for sequences of Caccioppoli sets, along subsequences we obtain the existence of limiting Caccioppoli sets $\Omega_{0}$ and $\Omega_1$ of the sets $\Omega^n_{0}$ and $\Omega^n_1$.
This is summarized in the next proposition.

\begin{prop}
\label{prop:conver}
Let $\{u_n\}_{n\in\N}$ be a sequence of lattice deformations with $u_n \in \mathcal{A}_n^{F_\lambda}$, satisfying the energy bound (\ref{eq:surfscaling}).
Let $\Omega^n_0$ be as above, then there exists a subsequence $\{n_j\}_{j\in \N}$ and (up to zero sets) disjoint Caccioppoli sets $\Omega_0$, $\Omega_1$ such that
\begin{align*}
&\Omega_0^{n_j}  \rightarrow \Omega_0 \mbox{ and }
\Omega_1^{n_j}  \rightarrow \Omega_1 \mbox{ in measure, i.e. }\\
&\big|\Omega_0^{n_j} \Delta \Omega_0\big|  \rightarrow 0 \mbox{ and }  \big|\Omega_1^{n_j} \Delta \Omega_1\big|  \rightarrow 0.
\end{align*}
Moreover, $\Omega= \Omega_0  \cap \Omega_1$.
\end{prop}

Finally, in the next proposition we relate the limiting sets $\Omega_0$ and $\Omega_1$ to corresponding limiting sets of ``low energy deformations'' (c.f. \ref{eq:surfscaling}) to the two-well problem.

\begin{prop}
\label{lem:L2_conv}
Let $\{u_n\}_{n\in \N} \subset \mathcal{A}_n^{F_\lambda}$ be piecewise affine functions on the grid $\Omega_n$ satisfying (\ref{eq:surfscaling}). Then, 
\begin{equation}
\label{eq:domconv1}
\begin{split}
\dist(\nabla u_n , SO(2)U_0) & \rightarrow 0 \mbox{ in } L^2(\Omega_0),\\ 
\dist(\nabla u_n , SO(2)U_1) & \rightarrow 0 \mbox{ in } L^2(\Omega_1).
\end{split}
\end{equation} 
More precisely, for $l\in\{0,1\}$
\begin{align*}
\int\limits_{\Omega_l} \dist^2(\nabla u_n, SO(2) U_l)dx\leq 100(c+1)H_n(u_n) + (100c)^2|\Omega_l \Delta \Omega_l^n|.
\end{align*}
\end{prop}
\begin{proof}
We only provide the proof for $l=0$, for $l=1$ the argument is analogous:
\begin{equation}
\label{eq:dist_wells}
\begin{split}
\int\limits_{\Omega_0} \dist^2(\nabla u_n, SO(2)U_0) d x 
&=\int\limits_{\Omega_0^n} \dist^2(\nabla u_n, SO(2)U_0) d x  + \int\limits_{\Omega_0\setminus \Omega_0^n} \dist^2(\nabla u_n, SO(2)U_0) d x\\
& \leq  H_n(u_n) + \int\limits_{\Omega_0\setminus \Omega_0^n} \dist^2(\nabla u_n, SO(2)U_0) d x.
\end{split}
\end{equation}
We continue by estimating the second term:
\begin{align*}
\int\limits_{\Omega_0\setminus \Omega_0^n} \dist^2(\nabla u_n, SO(2)U_0) d x
&\leq \int\limits_{\Omega_0\setminus \Omega_0^n} \chi_{\{|\nabla u_n|\leq 100\bar{c}\}} \dist^2(\nabla u_n, SO(2)U_0) dx\\
& \quad + \int\limits_{\Omega_0\setminus \Omega_0^n}\chi_{\{|\nabla u_n|\geq 100\bar{c}\}} \dist^2(\nabla u_n, SO(2)U_0) d x\\
& \leq (100\bar{c})^2|\Omega_0\Delta \Omega_0^n| \\
& \quad + 100 \bar{c} \int\limits_{\Omega_0\setminus \Omega_0^n} \dist^2(\nabla u_n, SO(2)U_0 \cup SO(2)U_1) d x\\
& \leq (100\bar{c})^2|\Omega_0\Delta \Omega_0^n| + 100 \bar{c} H_n( u_n).
\end{align*}
Inserting this back into (\ref{eq:dist_wells}) yields the desired bound.
\end{proof}
\begin{remark}
\label{rmk:set1}
We observe that the convergence $|\Omega_0 \Delta \Omega_0^n|\go 0$ (which follows as a consequence of the discussion of the spin system given in Appendix A) can be arbitrarily slow. Indeed, as an example, one could consider a finite number of stripes, in which $\nabla u_n \in SO(2)U_1$, with size $n^{\alpha}$ for any arbitrary $\alpha \in (0,1)$.
\end{remark}

\section{Second Derivative Control}
\label{sec:2derivative}
In this section we derive an important property of the Hamiltonian $H_n$. Although this is not directly used in our argument which leads to the $\Gamma$-limit of Theorem \ref{Th1}, this property in part explains the comparability of the continuous model from \cite{CS06} and our discrete model. In fact the Hamiltonian from Definition \ref{defi:HamiltonianII} does not only control the deviation of the gradients of $u$ from the energy wells, but also the discrete \emph{second} derivatives of $u$.
\\

In the sequel, we use $C$ to denote a universal constant which only depends only on $a$ and $b$ and may change from line to line.

\begin{lem}[Second derivative control]
\label{lem:second}
Let $u \in \mathcal{A}_n^{F_\lambda}$. Then there exists a constant $C=C(a,b)>0$ such that
\begin{align*}
n^2 |u^{i+1,j} + u^{i-1,j} - 2u^{i,j}| &\leq C n\sqrt{h^{i,j}_u},\\
n^2 |u^{i,j+1}+ u^{i,j-1} - 2u^{i,j}| &\leq C n\sqrt{h^{i,j}_u},\\
n^2 |u^{i+1,j+1} - u^{i,j+1}  - u^{i+1,j} + u^{i,j}| &\leq C n \left(\sqrt{h^{i,j}_u}+ \sqrt{h^{i+1,j}_u}+\sqrt{h^{i-1,j}_u} \right).
\end{align*}  
More concisely, we will also abbreviate this as $|\nabla^2_n u| \leq C n \sqrt{h^{i\pm 1, j\pm 1}_u}$, where $\nabla^2_n $ denotes the tensor of second finite differences of $u$.
\end{lem}

\begin{proof}
We recall that the density of the Hamiltonian can be rewritten in terms of the lengths and angles of the deformation (c.f. equation (\ref{eq:bracket1})). In the sequel, for notational convenience we will use the abbreviations $v^{i,j}:= \nabla_1 u^{i,j}$, $w^{i,j}:= \nabla_2 u^{i,j}$.\\ 

\emph{Step 1: Horizontal difference quotients.}
We begin by estimating the horizontal second order difference quotient. Here we distinguish two cases.\\
\emph{Step 1a: Smallness of $h$.}
In the first case we assume that $h^{i,j}_u \leq c$, where $c=c(a,b)>0$ is a constant that is much smaller than the distance between the wells (and much smaller than one). In particular, we may assume that one of the brackets in the definition of $h_u^{ij}$ is much smaller than one while the other is of the order $C(a,b)$, which is a universal constant that only depends on $a,b$. Without loss of generality, we assume that the first bracket in the definition of $h_u^{ij}$, i.e. (\ref{eq:bracket1}) is the small one.
Thus, we obtain that for some constant $C=C(a,b)$
\begin{align*}
 (|v^{i,j}|-|v^{i-1,j}|)^2 \leq 4((|v^{i,j}|-a)^2 + (|v^{i-1,j}|-a)^2) \leq Ch^{i,j}_u.
\end{align*}
Moreover, we infer that
\begin{align*}
|(v^{i,j},w^{i,j})| + |(v^{i-1,j},w^{i,j})| \leq C h^{i,j}_u.
\end{align*}
However, by linear algebra and the non-interpenetration condition, this implies that
\begin{align*}
n^4 |u^{ij} + u^{i-1,j} - 2u^{i,j}|^2 = n^2 |v^{ij}-v^{i-1,j}|^2 \leq Cn^2h^{i,j}_u,
\end{align*}
which yields the desired estimate.\\
\emph{Step 1b: $h_u^{ij} \geq c$.} In the case that $h_u^{ij} \geq c>0$ for some fixed constant $c=c(a,b)$, we directly use the triangle inequality:
\begin{align*}
n^4 |u^{ij} + u^{i-1,j} - 2u^{i,j}|^2 &\leq 4n^2(|v^{i,j}|^2  + |v^{i-1,j}|^2) \\
&\leq  C n^2 h^{i,j}_u+ n^2 C(a,b) \leq C n^2 h^{i,j}_u.
\end{align*}
Here the last estimate follows from the lower bound assumption on $h$.
Combining the results of Step 1a and Step 1b thus yields the full control on the horizontal second difference quotient. A similar estimate holds true for the vertical second difference quotient.\\

\emph{Step 2: Mixed second order differences.}
Again we consider two cases now depending on the local energy density of two neighboring points. \\
\emph{Step 2a: Small local energy density.} We assume that for a sufficiently small constant $c=c(a,b)>0$ we have $h^{i,j} + h^{i+1,j} \leq c(a,b)$. In this case we may assume that $\nabla u^{i,j}$ and $\nabla u^{i+1,j}$ are sufficiently close to a common energy well, i.e. we may for instance assume that the first bracket in the definition of $h_u$ is controlled in terms of $h_u$ for both points $(i,j)$ and $(i+1,j)$ (this follows directly from the definition of $\tilde{h}_u$ but can also be inferred from Lemma  \ref{lem:lower}. Then,
\begin{align*}
(|w^{i,j}| -a)^2 &\leq h^{i,j}_u \mbox{ and } (|w^{i+1,j}| -a)^2 \leq h^{i+1,j}_u,
\end{align*}  
which implies that
\begin{align}
\label{eq:lengths}
(|w^{ij}|- |w^{i+1,j}|)^2 \leq  h^{i,j}_u+h^{i+1,j}_u.
\end{align}
Moreover,
\begin{equation}
\label{eq:angles}
\begin{split}
(|v^{i,j}|- b)^2 & \leq  h^{i,j}_u + h^{i+1,j}_u,\\
|(w^{i,j}, v^{i,j})|& \leq h^{i,\,j}_u,\\
|(w^{i+1,j}, v^{i,j})|& \leq h^{i,j}_u+h^{i+1,j}_u.
\end{split}
\end{equation}
Arguing as above, by linear algebra and the non-interpenetration condition (which can also be interpreted as a condition on the orientation of the image triangles), the combination of (\ref{eq:lengths}) and (\ref{eq:angles}) implies that
\begin{align*}
n^4 |u^{i+1,j+1} - u^{i,j+1}  - u^{i+1,j} + u^{i,j}|^2 &\leq n^2 |w^{ij}- w^{i+1,j}|^2 \\
& \leq n^2( h^{i,\,j}_u+h^{i+1,j}_u),
\end{align*}
which is the desired result.\\
\emph{Step 2b: Large local energy.} We assume that $h^{i,\,j} + h(\nabla u^{i+1,j})\geq c(a,b)>0$. As in Step 1a, we directly conclude by using the triangle inequality:
\begin{align*}
& n^4 |u^{i+1,j+1}-u^{i+1,j}-u^{i,j+1}+u^{ij}|^2 \leq 4 n^4(|u^{i+1,j+1} - u^{i+1,j}|^2 + |u^{i,j+1}-u^{ij}|^2)\\
& \quad \leq n^2(4(|w^{ij}|-a)^2 + 8(a^2 + b^2) + 4(|w^{i+1,j}|-a)^2)\\
& \quad \leq 4 n^2  (h^{i,\,j}_u +h^{i+1,j}_u) + n^2 C(a,b)\\
& \quad \leq C(a,b)n^2 (h^{i,\,j}_u + h^{i+1,j}_u).
\end{align*}
This concludes the proof.
\end{proof}

\section{Sketch of Proof of the Discrete Coarea Formula}
\label{app:Coarea}

In this section, we give a (very rough) sketch of the proof of the discrete coarea formula. More precisely, we show that for any grid function $f^{ij}:\Omega_n \rightarrow \mR$ the following holds: 
\begin{align}
\int\limits_{0}^{\infty}\Per_M(\{(i,j) :\,f^{ij}\geq t\})dt \leq C \sum\limits_{(i,j)\in nM} \frac{1}{n^2} |\nabla_n f^{ij}|.
\lb{eq:coarea1}
\end{align}
Similar as in the continuous case (c.f. \cite{EvGa15}) this is a consequence of an integration by parts argument in combination with the bathtub principle: We consider the super-level sets $E_t:= \{(i,j): f^{ij}\geq t\}$ associated with the function $f^{ij}$. Let  $\sigma^{ij}= (\sigma^{ij}_1, \sigma^{ij}_2):\Omega_n \rightarrow \mR^2$ be any test function with $|\sigma^{ij}|\leq 1$. Then we have
\begin{equation}
\label{eq:coarea_pr}
\begin{split}
&\sum\limits_{(i,j)\in n M} \frac{1}{n^2}\left[n(f^{i+1,j}-f^{i,j})\sigma_{1}^{i,j} + n(f^{i,j+1}-f^{i,j})\sigma_{2}^{i,j}  \right] \\
&=
\int\limits_{0}^{\infty} \sum\limits_{(i,j)\in \partial (n E_t)}\frac{1}{n} (\sigma^{i,j}, \nu^{i,j}) dt.
\end{split}
\end{equation}
Here $\nu(x): \partial (n E_t) \rightarrow \mR^2$ denotes ``the outer unit normal field'' to $\partial (n E_{t})$. As $\partial (n E_t)$ is only piecewise affine, we define it as the classical outer unit normal field at all points at which the boundary is $C^1$. Due to the choice of our interpolation, the only possibility of violating the $C^{1}$ condition for the boundary is by forming corners of $45^{\circ}, 90^{\circ}, 180^{\circ}, 225^{\circ}, 270^{\circ}$. These corners are given as the intersection of two $C^{1}$ curves which are tangential to two grid edges. Hence, at these corner points $(i,j)$ we define $\nu^{ij}=  \left(\lim\limits_{x_1 \rightarrow (i,j) }\nu(x_1) + \lim\limits_{x_2 \rightarrow (i,j)}\nu(x_2)\right)$, where $x_1,x_2$ are points approaching the corner along the two intersecting grid edges.\\ 
We note that the left hand side of (\ref{eq:coarea_pr}) is clearly bounded from above by
\begin{align*}
\sum\limits_{(i,j)\in n M} \frac{1}{n^2}\left[n|f^{i+1,j}-f^{i,j}| + n|f^{i,j+1}-f^{i,j}| \right].
\end{align*}
Hence,
\begin{align*}
\int\limits_{0}^{\infty}\sum\limits_{(i,j)\in \partial (n E_t)} \frac{1}{n} (\sigma^{i,j},\nu^{i,j})dt \leq \sum\limits_{(i,j)\in n M} \frac{1}{n^2}\left[n|f^{i+1,j}-f^{i,j}| + n|f^{i,j+1}-f^{i,j}| \right].
\end{align*}
Choosing $\sigma^{ij}$ such that $(\sigma^{i,j},\nu^{i,j})=1$ on $\partial (n E_t)$, proves (\ref{eq:coarea1}).

\section{Proof of the Well-Definedness of the Algorithm for the Perturbed Grid Construction}
\label{app:alg}

In this section we present a proof of the well-definedness of the algorithm which yields the new grid in Step 4a of the Proof of Proposition \ref{prop:reduc}:

\begin{lem}
\label{lem:alg}
The algorithm in Step 4a of the Proof of Proposition \ref{prop:reduc} is well-defined, i.e. it is possible to choose the parameter $\theta>0$ such that there exists a number $c_{\theta}>0$ with the property that the volume fraction of possible choices in each step is non-empty and bounded from below: 
\begin{align}
\label{eq:welldef}
|\{ \text{possible choices in step } m\}| \geq c_{\theta} |B_m |.
\end{align}
\end{lem}

The main difficulty here is to ensure that in the selection of the \emph{possible choices in step $m+1$} not too many points are deleted. In particular, we have to ensure, that although for each pair $B_m, B_{m+1}$ there always is a volume fraction of $(1-\theta)|B_m \times B_{m+1}|$ rigid pairs, these pairs do not involve too many points which had to be deleted during one of the previous $m$ steps. If this were the case, it could in principle occur that the algorithm terminates without having constructed a new grid. 

\begin{figure}[t]
\centering
\includegraphics[width=0.9\textwidth]{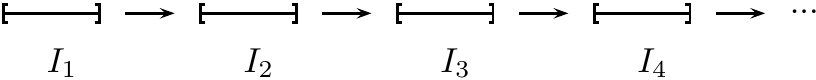}
\caption{The one-dimensional chain of intervals described in step 1. Without loss of generality we may assume that the intervals are given as $I_n=[n-1,n]$ with $n\in \N$. The arrow between the intervals indicates the neighboring relation. In a \emph{chain of intervals} each of the intervals only has one neighbor.}
\label{fig:chain}
\end{figure}

\begin{proof}
We divide the proof into two steps and first show the analogous result in the setting of one-dimensional intervals of equal length that form a one-dimensional chain (c.f. Figure \ref{fig:chain}). In the second step we then show that the our algorithm essentially reduces to the previous setting.\\

\begin{figure}[t]
\includegraphics[width=0.47\textwidth]{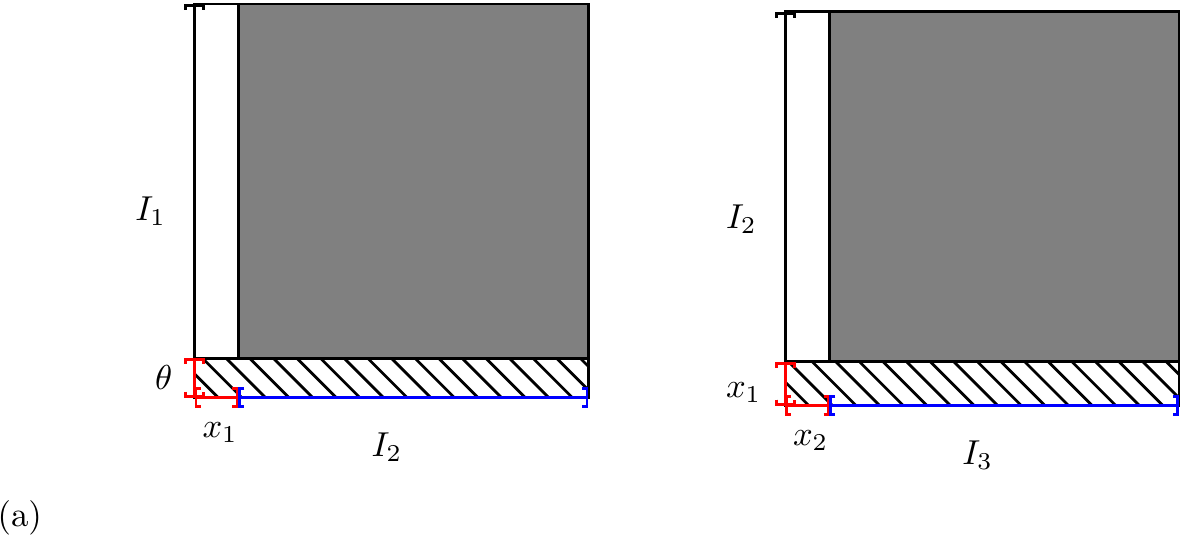}
\includegraphics[width=0.47\textwidth]{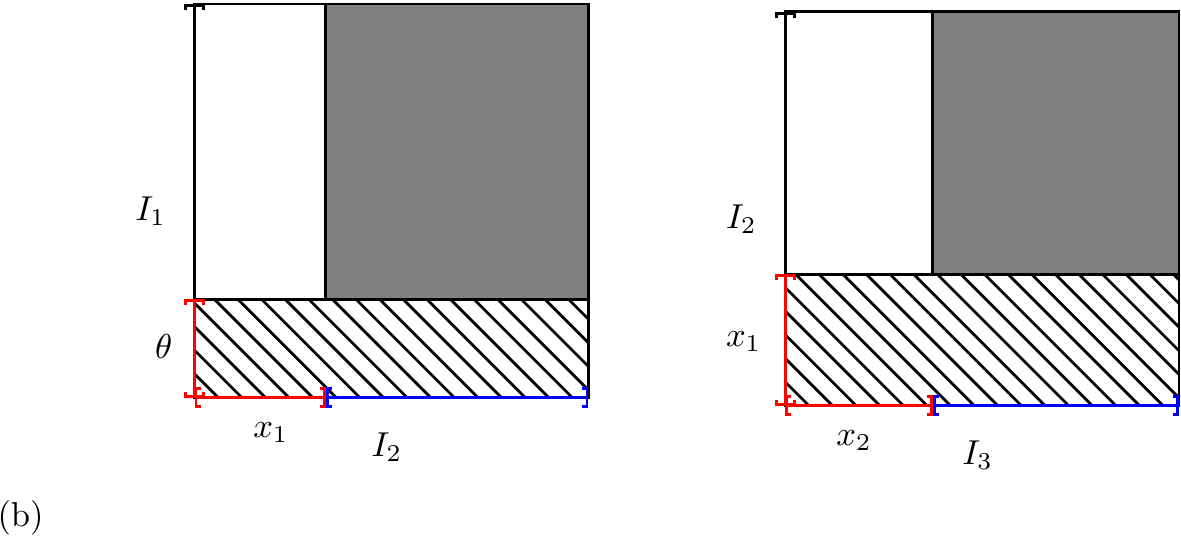}
\caption{\small{The worst case scenario described in Step 1 in the proof of Lemma \ref{lem:alg}. Although the overall volume fraction of rigid points between two intervals $I_m$, $I_{m+1}$ is always given by $(1-\theta)|I_m \times I_{m+1}|$, it is in principle possible that the \emph{effectively} useful volume fraction in step $m$ of rigid pairs in the algorithm is smaller than that. This is due to the possibility that a lot of rigid points are given by pairs $(x,y)$ with $x\in I_{m}$, $y\in I_{m+1}$ such that $x\in I_{m}^{\text{bad}}$, i.e. that many points $x$ lie in the set which had to be deleted from being a possible choice in a previous step of the algorithm. \\
In the above figure this is indicated for $\theta = 0.1$ (a) and $\theta = 0.25$ (b). Here the rigid points are given by the union of the dashed and the gray rectangles. The white rectangle is the set of non-rigid points between the intervals. In the first step, the intervals $I_0$ and $I_1$ have a volume fraction of $(1-\theta)|I_0 \times I_1|$ rigid pairs. Hence after deleting the non-rigid pairs; a volume fraction of at least $\theta$ points of $I_1$ is rigid. However, a volume fraction of up to $\theta$ of $I_1$ consists of bad points, which had to be deleted. In the next step, which is here depicted as the first schematic illustration in (a), the intervals $I_1,I_2$ still have a volume fraction of $(1-\theta)|I_1 \times I_2|$ rigid points. These are schematically depicted as the union of the dashed and the gray rectangles. However, it is possible that a large amount of rigid pairs involve points in $I_1^{\text{bad}}$. These are indicated as the points in the dashed rectangle (the red interval in the vertical axis corresponds to the bad points $I_1^{\text{bad}}$ of $I_1$. Thus, in order to construct the new perturbed grid, only the gray pairs are of use. However, this implies that in $I_2$ the red interval of length $x_1^{\theta}$ is deleted by the algorithm, yielding a bad set $I_2^{\text{bad}}$ of length $x_1^{\theta}:= \frac{\theta}{1-\theta}$). In the next step, which is depicted in the second illustration in (a), this could happen again: Again a large fraction of rigid points is given by pairs in which the first component lies in $I_2^{\text{bad}}$. Thus, again, only the gray square can be used as rigid points in order to construct the new grid. Thus, $|I_2^{\text{bad}}|=: x_2^{\theta}= \frac{\theta}{1-x_1^{\theta}}$}. \\
As (a) indicates, for $\theta = 0.1$ the sequence $x_m^{0.1}$ converges quite fast to the value $\bar{x}^{0.1} \sim 0.112702$. The figure (b) on the right depicts the same scenario for the $\theta = 0.25$. This is the largest value of $\theta$ for which convergence still holds with $\bar{x}^{0.25}=0.5$. For larger values of $\theta$ the algorithm terminates before having created a new grid.}
\label{fig:step1}
\end{figure}

\emph{Step 1:} We argue in the context of one-dimensional intervals of equal size, which form a one-dimensional chain (c.f. Figure \ref{fig:chain}). More precisely, we imagine that we have a partition of the real half-line into intervals $I_n:=[n,n+1]$ for $n\in \N\cup\{0\}$. The end-points are the vertices of our one-dimensional grid. We assume that we have an analogue of Proposition \ref{prop:twell}, asserting that for each neighboring pair $I_n, I_{n+1}$ of intervals, the volume of rigid pairs $(x,y)$ of points in $I_n \times I_{n+1}$ has volume at least $(1-\theta)|I_n \times I_{n+1}|$. We apply the algorithm described in Step 4a of the proof of Proposition \ref{prop:reduc} to this one-dimensional ``chain of intervals'' (where $B_m$ is replaced by $I_m$). We claim that this algorithm is well-defined in the sense of (\ref{eq:welldef}) and hence does not terminate before having constructed a new perturbed grid.\\
To this end, we prove that at any stage of the algorithm the points which are rigid in $I_m$ and are \emph{possible choices in step $m+1$} never become the empty set. On the contrary, we show that they satisfy the bound (\ref{eq:welldef}). We claim that this is true, since $\theta$ (which is fixed throughout the algorithm and in particular does not depend on the step $m$) can be chosen sufficiently small. Indeed, for each pair $I_m$, $I_{m+1}$ the volume of rigid pairs $(x,y)\in I_m \times I_{m+1}$ is $(1-\theta)|I_m \times I_{m+1}|$. However, this might be diminished by the points which have been removed in the previous steps of the algorithm (as it could be the case that all of these deleted points were rigid points for the next step, c.f. Figure \ref{fig:step1}). Hence the effective volume of rigid pairs could be smaller than  $(1-\theta)|I_m \times I_{m+1}|$. In this context the worst case scenario is given by the following setting: All points of $I_{m+1}$ form rigid pairs with the ``bad set'' $I_m^{\text{bad}}$ of $I_{m}$, i.e. with those points which were removed from being possible choices in the previous steps of the algorithm, and the volume of the points in $I_{m+1}$ which form rigid pairs with points in $I_m\setminus I_m^{\text{bad}}$ is minimized (c.f. Figure \ref{fig:step1}). 
In the first step of the algorithm the ``bad set'' $I_m^{\text{bad}}$ can be of measure at most $\theta$. However, in principle this could increase in the next steps. We have to show that the ``bad set'' remains small (depending on $\theta$). Indeed, following the description of the previous worst case scenario, we estimate the ``bad set'' $I_m^{\text{bad}}$. Computing the volume of the bad set in step $m$ (c.f. Figure \ref{fig:step1}), we infer that the bad set $I_m^{\text{bad}}$ is bounded by the solution of the recursion relation
\begin{align*}
x_m^{\theta} = \frac{\theta}{1-x_{m-1}^{\theta}}, \ x_0^{\theta}=\theta.
\end{align*}
However, if $\theta\leq \frac{1}{4}$, this recursion relation converges to the limit $\bar{x}^{\theta}:= \frac{1}{2}(1-\sqrt{1-4\theta})$, due to monotonicity. Hence, the bad set remains bounded. This yields our claim (\ref{eq:welldef}).\\

\emph{Step 2:} We claim that the previous steps implies that the grid construction algorithm in the proof of Proposition \ref{prop:reduc} works for our two-dimensional set-up. First we note that the fact that the balls become smaller does not matter, as the result only depends on the respective volume fractions involved. Secondly, we note that also in the two-dimensional case, we have to rule out the worst case scenario of an accumulating ``bad set'' and that the recursion relation from above would still give the desired bound if there was only a single neighbor to each vertex. Finally, we notice that the presence of \emph{finitely} many neighboring vertices (instead of having a single vertex only) does not change the convergence of the algorithm (if $\theta$ is decreased according to the number of possible neighboring vertices). This is a consequence of the fact that in comparison to the setting in step 1, only finitely many times additional points are deleted in the presence of several neighbors.
\end{proof}

\end{appendix}

\bibliography{bibliography}
\bibliographystyle{alpha} 
\clearpage
\addtocounter{tocdepth}{2}

\end{document}